\documentclass[leqno,11pt,a4paper]{amsart}
\usepackage[full]{textcomp}
\usepackage{newpxtext}
\usepackage{cabin} % sans serif
\usepackage[varqu,varl]{inconsolata} % sans serif typewriter
\usepackage[bigdelims,vvarbb]{newpxmath} % bb from STIX
\usepackage[cal=cm,bb=esstix,bbscaled=1.05]{mathalfa} % mathcal
\usepackage[margin=2.4cm]{geometry}
\usepackage{savesym}

\usepackage{mathabx}
\usepackage[usenames,dvipsnames]{color}
\usepackage{hyperref}
\hypersetup{
 colorlinks,
 linkcolor={red!50!black},
 citecolor={blue!50!black},
 urlcolor={blue!80!black}
}
\usepackage{tikz}
\usetikzlibrary{cd}
\usepackage{graphicx}
\usepackage[all,cmtip]{xy}
\usepackage{mleftright}
\usepackage{booktabs}
\usepackage{cite}
\usepackage{enumitem}
\usepackage{mathdots}
\setlist[enumerate]{labelsep=*, leftmargin=1.5pc}
\setlist[enumerate]{label=\normalfont(\roman*), ref=\roman*}
\newtheorem{thm}{Theorem}[section]
\newtheorem{lemma}[thm]{Lemma}
\newtheorem{cor}[thm]{Corollary}
\newtheorem{prop}[thm]{Proposition}

\newtheorem{conjecture}[thm]{Conjecture}
\theoremstyle{definition}
\newtheorem{example}[thm]{Example}
\newtheorem{remark}[thm]{Remark}
\newtheorem{definition}[thm]{Definition}
\newtheorem{question}[thm]{Question}

\newtheorem{conventions}[thm]{Conventions} 
\numberwithin{equation}{section}
\newcommand{\Z}{\mathbb{Z}}
\newcommand{\Q}{\mathbb{Q}}
\newcommand{\R}{\mathbb{R}}
\newcommand{\C}{\mathbb{C}}

\newcommand{\cN}{\mathcal{N}}

\newcommand{\bK}{\mathbb{K}}

\newcommand{\mO}{\mathcal{O}}

\DeclareMathOperator{\GL}{GL}
\newcommand{\pr}{\mathbb{P}}

\newcommand{\intr}[1]{{#1}^\circ}

\renewcommand{\dim}[1]{\operatorname{dim}\mleft({#1}\mright)}

\renewcommand{\min}[1]{\operatorname{min}\mleft\{{#1}\mright\}}

\newcommand{\on}{\operatorname}
\newcommand{\ol}{\overline}
\newcommand{\wt}{\widetilde}

\newcommand{\mat}{\left(\begin{array}}
\newcommand{\tam}{\end{array}\right)}
\newcommand{\NE}{\ol{\text{\textnormal{NE}}}}

\newcommand{\calg}{c^\text{\textnormal{alg}}}
\newcommand{\cech}{c^\text{\textnormal{ECH}}}

\newcommand{\lf}{\lfloor}
\newcommand{\rf}{\rfloor}
\newcommand{\lc}{\left\lceil}
\newcommand{\rc}{\right\rceil}
\newcommand{\ip}{\text{\textnormal{IP}}}
\newcommand{\neb}{\ol{\on{NE}}}
\newcommand{\SW}{\on{SW}}
\renewcommand{\P}{\mathbb{P}}
\newcommand{\Ind}{\on{I}}
\newcommand{\CZ}{\on{CZ}}
\newcommand{\ECH}{\on{ECH}}

\newcommand{\Gr}{\on{Gr}}

\definecolor{dark green}{rgb}{0.0, 0.6, 0.0}

\graphicspath{{images/}}
\begin{document}
%-------------------------------------------------------------------------------
\author[J.~Chaidez]{J.~Chaidez}
\address{Department of Mathematics\\University of California at Berkeley\\Berkeley, CA\\94720\\USA}
\email{jchaidez@berkeley.edu}
\author[B.\,Wormleighton]{B.~Wormleighton}
\address{Department of Mathematics\\Washington University in St.~Louis\\St.~Louis, MO\\63130\\USA}
\email{benw@wustl.edu}
%-------------------------------------------------------------------------------
\title[ECH Embedding Obstructions For Rational Surfaces]{ECH Embedding Obstructions For Rational Surfaces}
\maketitle
%-------------------------------------------------------------------------------

\begin{abstract}
Let $(Y,A)$ be a smooth rational surface or a possibly singular toric surface with ample divisor $A$. We show that a family of ECH-based, algebro-geometric invariants $\calg_k(Y,A)$ proposed in \cite{bwo} obstruct symplectic embeddings into $Y$. Precisely, if $(X,\omega_X)$ is a $4$-dimensional star-shaped domain and $\omega_Y$ is a symplectic form Poincar\'e dual to $[A]$ then
\[(X,\omega_X)\text{ embeds into }(Y,\omega_Y)\text{ symplectically } \implies \cech_k(X,\omega_X) \le \calg_k(Y,A)\]
We give three applications to toric embedding problems: (1) these obstructions are sharp for embeddings of concave toric domains into toric surfaces; (2) the Gromov width and several generalizations are monotonic with respect to inclusion of moment polygons of smooth (and many singular) toric surfaces; and (3) the Gromov width of such a toric surface is bounded by the lattice width of its moment polygon, addressing a conjecture of Averkov--Hofscheier--Nill in \cite{ahn}.
\end{abstract}

\section{Introduction} \label{sec:intro} A symplectic embedding of symplectic manifolds $(X,\omega) \to (X',\omega')$ of the same dimension is a smooth embedding $\varphi:X \to X'$ that intertwines the symplectic form, i.e. $\varphi^*\omega' = \omega$. The study of symplectic embeddings has been a major topic in symplectic geometry ever since Gromov proved his eponymous non-squeezing theorem, stating that
\[ B^{2n}(r) \text{ symplectically embeds into } B^2(R) \times \C^{n-1} \quad \iff \quad r \le R\]
Symplectic capacities provide the primary tool for obstructing symplectic embeddings. Roughly speaking, a symplectic capacity $c$ is a numerical invariant associated to a symplectic manifold (usually in a restricted class, e.g. exact) such that $c(X) \le c(X')$ whenever $X$ symplectically embeds into $X'$. The most famous example is the \emph{Gromov width} of $X$, defined by
\begin{equation} c_G(X) := \sup\{\pi r^2 \; : \; B(r) \text{ symplectically embeds into }X\}\end{equation}
Capacities like $c_G$ have been used to great effect to provide complete solutions to many symplectic embedding problems.

\vspace{4pt}

One family of capacities that have been applied with particular success in dimension 4 are the \emph{ECH capacities} $\cech_k$ (one for each integer $k \ge 1$) introduced by Hutchings in \cite{mh4}. These capacities are defined using embedded contact homology (or ECH for short), a version of Floer homology for contact $3$-manifolds with a deep connection to Seiberg-Witten theory. They also provide sharp embedding obstructions for several 4-dimensional symplectic embedding problems, such as ellipsoids into ellipsoids \cite{md} and (more generally) of concave toric domains into convex toric domains \cite{cg}. This paper is about symplectic embedding obstructions derived using ECH.

\subsection{ECH capacities via algebraic geometry} Our present story begins with the work of Wormleighton (the second author of this paper) in \cite{bwo}, which we now review in some detail.

\vspace{4pt}

Recall that a \emph{toric domain} $X_\Omega$ is the inverse image $\mu^{-1}(\Omega)$ of a compact subset $\Omega \subset [0,\infty)^2$ with open interior under the standard moment map on $\C^2$.
\[\mu:\C^2 \to \R^2 \qquad (z_1,z_2) \mapsto (\pi |z_1|^2,\pi |z_2|^2)\]
The region $\Omega$ is called the \emph{moment image}. A toric domain $X_\Omega$ is \emph{convex} if $\Omega = K \cap [0,\infty)^2$ where $K \subset \R^2$ is a convex set and $0 \in K$. Likewise, $X_\Omega$ is \emph{concave} if $\Omega = C \cap [0,\infty)^2$ where $\R^2 \setminus C$ is convex and $0 \in C$. Finally, a \emph{rational} toric domain is a convex toric domain where $\Omega$ is the convex hull of finitely many rational points in $[0,\infty)^2$.

\vspace{4pt}

The ECH capacities of toric domains have been studied extensively (c.f. \cite{mh2,mh3,ccfhr,cg}). For rational toric domains, the ECH capacities can be combinatorially computed using the moment polytope $\Omega$, and these computations bear a remarkable resemblance to calculations arising in the algebraic geometry of $\Q$-line bundles over toric surfaces. This observation was first leveraged (for ellipsoids) in the work of Cristofaro-Gardiner--Kleinman \cite{cgk}. In \cite{bwo}, Wormleighton formalized it as a theorem.

\vspace{4pt}

To state this theorem we observe that, given a moment polytope $\Omega$, there is in addition to $X_\Omega$, an associated projective algebraic surface $Y_\Omega$ described by the inner normal fan of $\Omega$. This surface can be singular, and may alternately be viewed as a toric, symplectic orbifold with moment polytope $\Omega$. It comes equipped with a canonical ample $\R$-divisor $A_\Omega$ on $Y_\Omega$. 

\begin{thm}[{\cite[Thm.~1.5]{bwo}}] \label{thm:bwo_main} Let $X_\Omega$ be a rational toric domain and $(Y_\Omega,A_\Omega)$ be the corresponding polarized toric surface. Then
\begin{equation} \label{eqn:bwo_main} \cech_k(X_\Omega) = \inf_{D\in\on{nef}(Y_\Omega)_\Q}\{D\cdot A_\Omega:h^0(D) \ge k + 1\}\end{equation}
\end{thm}

\noindent Here the infimum is over all nef $\Q$-divisors in $Y_\Omega$. For the more symplectically minded reader, a nef divisor may be thought of as a homology class that is represented by a disconnected $J$-curve, and which has non-negative intersection with any other $J$-curve. For example, in $\P^2$ this is every non-negative multiple of the hyperplane class $[\P^1]$, while in $\P^1 \times \P^1$ this is every non-negative combination of $[\P^1 \times \text{pt}]$ and $[\text{pt} \times \P^1]$.

 Theorem \ref{thm:bwo_main} allows one to leverage the computational tools developed for toric geometry to perform calculations, and implies a number of nice results about the asymptotics of the ECH capacities as $k \to \infty$. See \cite{bwo} for more results.

\subsection{Geometric explanation} \label{subsec:intuitive_argument} The proof of Theorem \ref{thm:bwo_main} in \cite{bwo} is largely combinatorial, and amounts to checking that the two quantities agree using previously known explicit formulas. Thus, it is natural to wonder if there is some deeper geometric phenomenon at play. We now sketch a heuristic argument suggesting that this is indeed the case. 

\vspace{4pt}

To start, given a moment polytope $\Omega$, we observe that the surface with divisor $(Y_\Omega,A_\Omega)$ and domain $X_\Omega$ are related. Indeed, the interior $X^\circ_\Omega$ of $X_\Omega$ and the complement $Y_\Omega \setminus A_\Omega$ are equivariantly symplectomorphic and one can write down a ``collapsing map'' $\pi:\partial X_\Omega \to A_\Omega$ whose fibers are generically circles. If $Y_\Omega$ is smooth, we can (roughly speaking) write 
\begin{equation} \label{eqn:Y_X_N_splitting} Y_\Omega = X_\Omega \cup_Z N_\Omega\end{equation}
where $N_\Omega$ is a very thin neighborhood of $A_\Omega$ and $Z$ is the boundary of $N_\Omega$. Thus we have the following picture.

\vspace{-10pt}
\begin{figure}[h]
\centering
\caption{The relationship between $Y_\Omega$ and $X_\Omega$.}
\vspace{10pt}
\includegraphics[width=.8\textwidth]{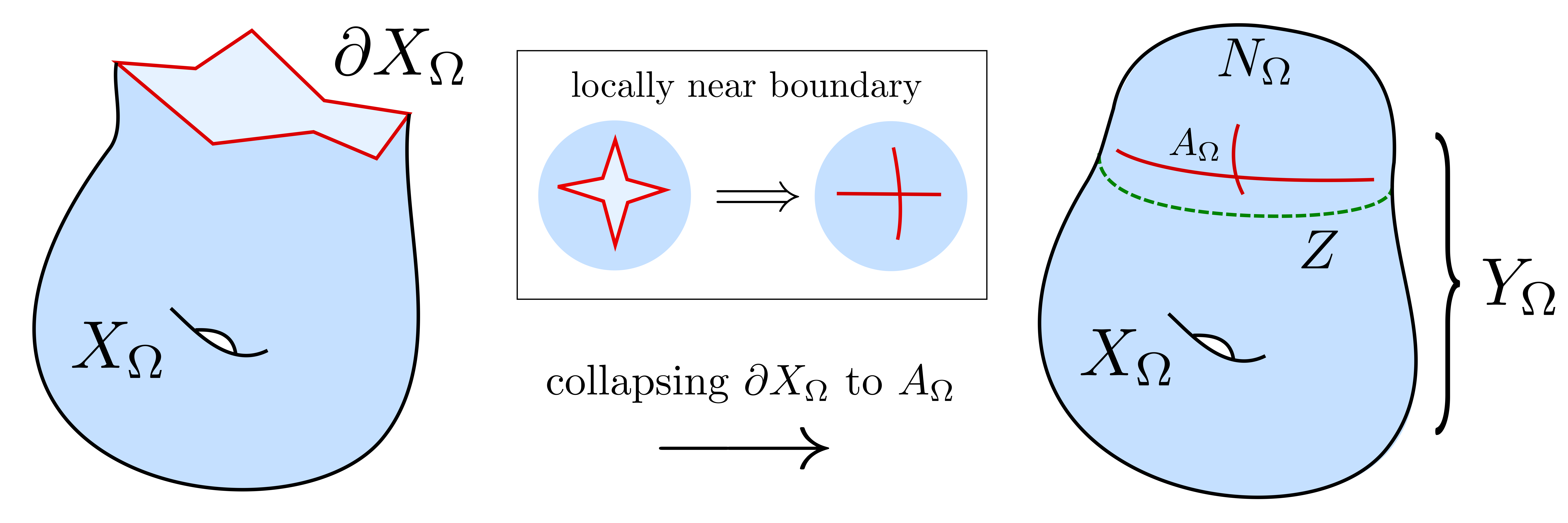}
\label{fig:Xomega_and_Yomega}
\end{figure}

Now we return to a discussion of capacities. Dissecting the construction of $\cech_k$, we find that the 1st ECH capacity of $X_\Omega$ is (again, roughly speaking) computed as the minimum area of certain disconnected holomorphic curves $u$ in $\widehat{Z} = \R \times Z$ satisfying some conditions. First, each component $C$ of $u$ is embedded, cylindrical at $\pm\infty$ and comes with an integer weight $n_C \in \Z_+$. Second, $u$ must pass through a point $p \in \widehat{Z}$ (fixed for all $u$). The $k$th ECH capacity of $X_\Omega$ is given by sequences $u_i$ of $k$ such curves with matching ends at $\pm \infty$.

\vspace{4pt}

One way that sequences $u_i$ of this form arise naturally is by neck stretching $Y_\Omega$ along the hypersurface $Z$. Namely, a disconnected curve $D \subset Y_\Omega$ with embedded components that is equipped with integer weights on its components and that passes through $k$ generic points in $Y_\Omega$ will (if it survives the stretching process) produce a sequence $u_i$ as above.

\vspace{-10pt}
\begin{figure}[h]
\centering
\caption{Neck stretching divisors to acquire ECH curves.}
\vspace{10pt}
\includegraphics[width=.8\textwidth]{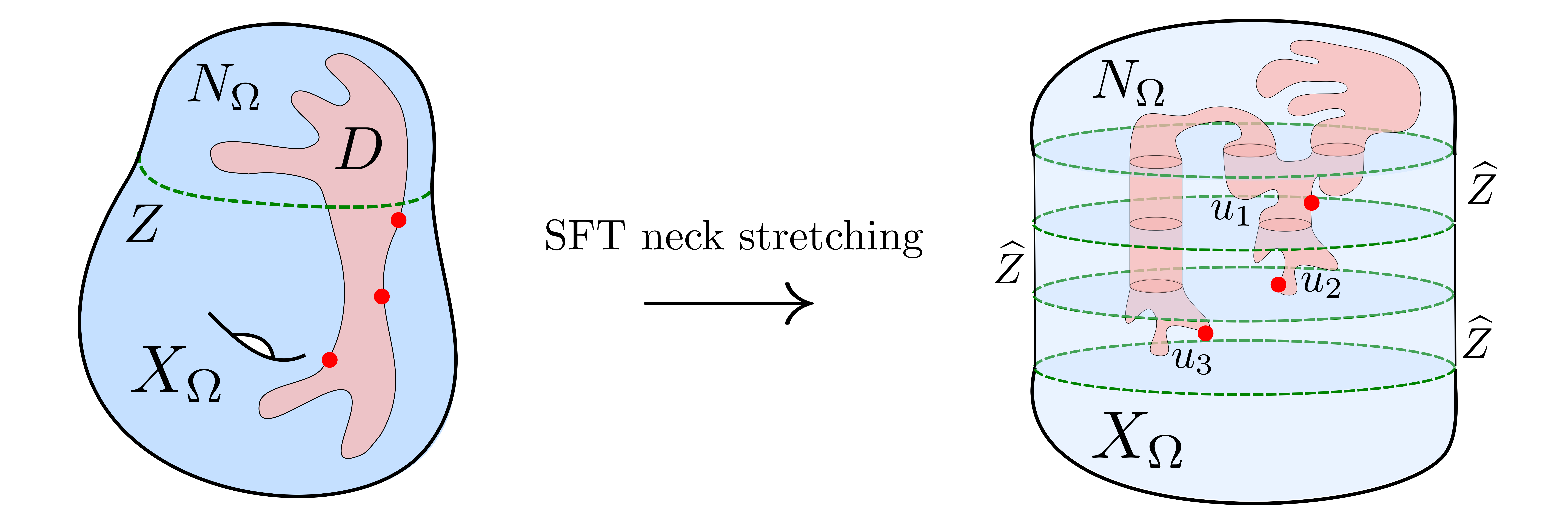}
\label{fig:intro_neck_stretching}
\end{figure}

\noindent The curve $D$ is essentially an effective, integral Weil divisor. If $D$ passes through $k$ points, then we expect the moduli of divisors $\mathcal{M}_D$ in the class of $D$ to satisfy $\dim{\mathcal{M}_D} \ge 2k$. Furthermore, the area of $D$ in $Y_\Omega$ is given by $A_\Omega \cdot D$ since $A_\Omega$ is Poincare dual to the Kahler form on $Y_\Omega$.

\vspace{4pt}

The above discussion leads us to expect an inequality of the following form, which strongly resembles one direction of the equality (\ref{eqn:bwo_main}).
\[
\cech_k(X_\Omega) \le \text{min}\{ A_\Omega \cdot D \; | \;\text{effective divisors $D$ with $\dim{\mathcal{M}_D} \ge 2k$ + more (?))}\}
\] 
Note that, in the above discussion, we did not reference the fact that $X_\Omega$ and $Y_\Omega$ arose via toric geometry or that $X_\Omega = Y_\Omega \setminus A_\Omega$. In fact, the entire argument seems sensible if $(Y,A)$ is an arbitrary projective surface with ample divisor and $X \subset Y$ is an embedded exact symplectic sub-domain.

\begin{remark} A more precise perspective on the curve $D$ in $Y_\Omega$ is that it arises in the moduli space count used to define the \emph{Gromov-Taubes invariant} of a symplectic 4-manifold \cite{tb,md2}. This neck stretching phenomenon is, morally speaking, the reason that ECH is the Floer theory categorifying the Gromov-Taubes invariants. 

In practice, this fact is formalized using the isomorphism of ECH with a variant of Seiberg-Witten-Floer homology \cite{tb2}, and the equivalent of the Gromov-Taubes invariants with the Seiberg-Witten invariants \cite{tb3}. In order to make the discussion of this section (\S \ref{subsec:intuitive_argument}) rigorous, we will make use of these equivalences via a result of Hutchings (see Theorem \ref{thm:ECH_as_SFT} in \S \ref{subsec:ECH_to_SW}).
\end{remark}

\subsection{Main results} \label{subsec:main_results} We are now ready to state the main theorem of this paper, which formalizes the discussion of \S \ref{subsec:intuitive_argument}. First we recall the notion of algebraic capacity from \cite{bwo,bwt,bwr}.

\begin{definition}[Definition \ref{def:alg_cap}] The \emph{$k$th algebraic capacity} $\calg_k(Y,A)$ of a rational projective surface $Y$ with ample $\R$-divisor $A$ is 
\[\calg_k(Y,A):=\inf_{D\in\on{Nef}(Y)_\Z}\{D\cdot A:\chi(D)\geq k+\chi(\mO_Y)\}\]
Here $\on{Nef}(Y)_\Z$ denotes the set of nef $\Z$-divisors on $Y$.
\end{definition}

Recall that a \emph{star-shaped domain} $X \subset \C^2$ is a codimension $0$ sub-manifold with boundary possessing a point $p \in X$ with the property that any other point $q \in X$ is connected to $p$ by a line segment in $X$. We do not require $X$ to have smooth boundary.

\begin{thm} (Theorem \ref{thm:main}) \label{thm:main_intro} Let $X \to Y$ be a symplectic embedding of a star-shaped domain $X$ into a smooth rational projective surface $(Y,\omega_A)$ with a ample $\R$-divisor $A$ with $[\omega_A] = \text{PD}[A]$. Then
\begin{equation} \label{eqn:main} \tag{$\star$}
\cech_k(X)\leq\calg_k(Y,A)
\end{equation}
\end{thm}

\begin{remark} Methods of algebraic geometry have been applied extensively to symplectic embedding problems for rational and toric surfaces, and our result is just one more perspective on this story. We refer the reader to the work of McDuff \cite{md3}, McDuff-Polterovich \cite{mp}, Anjos-Lalonde-Pinsonnault \cite{alp}, Casals-Vianna \cite{cv} and Christofaro-Gardiner-Holm-Mandini-Pires \cite{chmp} for just a few examples. Likewise, rationality is a key assumption in many embedding results (even those that use purely symplectic methods). See, for example, the work of Buse-Hind \cite{bh} and Opshtein \cite{op}. Note that our references here are not at all exhaustive. \end{remark}

\begin{remark} The formula (\ref{eqn:main}) provides a new computational tool for studying the ECH capacities of star shaped domains living within divisor complements. Indeed, the nef cones of surfaces are very well studied and many structural results exist which may be brought to bear while studying $\cech$ via Theorem \ref{thm:main}. Furthermore, the nef cone is often polyhedral, and thus methods from convex optimization can be utilised to compute $\calg$. We hope to explore the combinatoral and computational implications of (\ref{eqn:main}) in future work.\end{remark}

\vspace{4pt}

Although we were originally motivated to prove Theorem \ref{thm:main_intro} in order to study non-toric surfaces, many interesting implications appear even in the toric setting. In particular, \cite[Thm.~1.5]{bwo} implies that the inequality in Theorem \ref{thm:main_intro} is an equality for certain divisor complements, and this is key to our applications. We will now discuss the three results on symplectic embeddings into smooth toric surfaces that we will prove. 

\vspace{4pt}

For our first application, we prove that these obstructions are sharp for embeddings of concave toric domains into toric surfaces.

\begin{thm} \label{thm:app_1} (Theorem \ref{thm:concave_to_tor_surface}) Let $X_\Delta$ be a concave toric domain with interior $X^\circ_\Delta \subset X_\Delta$, and let $(Y_\Omega,A_\Omega)$ be a smooth toric surface. Then
\[X^\circ_\Delta \text{ symplectically embeds into }Y_\Omega \quad \iff \quad \cech_k(X_\Delta) \le \calg_k(Y_\Omega,A_\Omega)\]
\end{thm}

\noindent This result uses a similar result of Christofaro-Gardiner in \cite{cg}, for embeddings of concave domains into convex domains. Theorem \ref{thm:app_1} essentially shows that the extra freedom provided by gluing the divisor $A_\Omega$ into $X^\circ_\Omega$ makes no difference for embeddings of concave domains.

\vspace{4pt}

For our next application, we prove the following result that includes a folk conjecture about the Gromov width. Let $\Xi$ be the moment polygon of a concave toric domain and define the $\Xi$\emph{-width} by
$$c_\Xi(X):=\sup\{r:X_{r\Xi}\text{ symplectically embeds in $X$}\}$$
When $\Xi$ is the triangle with vertices $(0,0),(1,0),(0,1)$ the $\Xi$-width $c_\Xi$ is just the Gromov width $c_G$.

\begin{thm} \label{thm:app_2} (Corollary \ref{co:gromov_monotonicity} + Corollary \ref{cor:Xi_width_monotonicity}) Let $\Xi$ be the moment polygon of a concave toric domain. Suppose $\Omega \subset \Delta$ is an inclusion of moment polytopes of smooth toric projective surfaces. Then
\[c_\Xi(Y_\Omega) \le c_\Xi(Y_\Delta)\]
In particular, the Gromov widths satisfy
\[c_G(Y_\Omega) \le c_G(Y_\Delta)\]
\end{thm} 

\noindent In fact, we prove Theorem \ref{thm:app_2} (and Theorem \ref{thm:app_1}) for all projective surfaces (even singular ones) that possess a smooth fixed point. Note that any smooth symplectic toric $4$-manifold is a smooth projective toric surface (c.f. \cite{et}) so Theorem \ref{thm:app_2} may be stated in those terms as well.

\begin{remark} There have been previous results (c.f. \cite[Thm. 1.2]{chmp}) indicating that a ball (or more generally, ellipsoid) embeds into a toric domain if and only if it embeds into the corresponding toric surface. These results are related to Theorem \ref{thm:app_2}, and can actually be used to recover some cases. See \S \ref{subsec:emb_to_toric} for more discussion. \end{remark}

\vspace{4pt}

Finally, we prove an estimate of the Gromov width of a toric surface in terms of the lattice width of its moment polygon. This result is \cite[Conjecture 3.12]{ahn}.

\begin{definition} The \emph{lattice width} $w(\Omega)$ of a moment polytope is defined by
\[w(\Omega) := \underset{l \in \Z^n \setminus 0}{\text{min}}\Big( \underset{p,q \in \Omega}{\text{max}} \; \langle l,p - q\rangle\Big)\]
\end{definition}

\begin{thm} \label{thm:app_3} (Corollary \ref{cor:gromov_combo_bound}) Let $\Omega$ be a moment polygon with a smooth vertex. Then
$$c_G(Y_\Omega) \le w(\Omega)$$
In particular, this holds when $\Omega$ is Delzant or, equivalently, when the toric surface $Y_\Omega$ is smooth.
\end{thm}

\noindent Theorem \ref{thm:app_3} follows from Theorem \ref{thm:app_2} and a rigorous version of a heuristic argument from \cite{ahn}. 

\begin{remark} The assumption that the moment polytope has a smooth vertex in Theorems \ref{thm:app_1}, \ref{thm:app_2} and \ref{thm:app_3} is an technical assumption that may be removable with different methods. \end{remark}

\subsection{Future directions} \label{subsec:future_directions} There are a number of interesting research directions along the lines of \cite{bwo} and this paper that are worth exploring. We will comment on these now.

\vspace{4pt}

First, Theorem \ref{thm:bwo_main} in \cite{bwo} gives an equality for the ECH capacities, and it is natural to ask when Theorem \ref{thm:main} can be upgraded to an equality as well. Here is a guess along those lines. 

\begin{conjecture}[ECH of divisor complements] \label{conj:equality} Let $(Y,\omega_A)$ be a rational projective surface with an ample $\R$-divisor $A$ such that $\on{sing}(Y) \subseteq\on{supp}(A)$ and suppose $Y\setminus\on{supp}(A)$ is deformation-equivalent to a ball. Then,
\[\cech_k(Y\setminus\on{supp}(A)) = \calg_k(Y,A)\]
\end{conjecture}

\noindent Note that $Y \setminus A$ can still be viewed as the interior of a star shaped domain with corners $X$. Proving Conjecture \ref{conj:equality} would require either a clever argument for packing $X$ or a very delicate understanding of the ECH and Reeb dynamics of smoothings of $X$. 

\vspace{4pt} 

Beyond the ECH capacities, there are finer obstructions defined (by Hutchings in \cite{mh3}) for embeddings of convex toric domains into other convex toric domains. These invariants are still poorly understood. The hope is that they could help solve some of the more obstinate embedding problems, such as the problem of embedding polydisks into ellipsoids.

\begin{question} Let $\Delta$ and $\Omega$ be rational moment polytopes. Is there a framework for treating the obstructions of \cite{mh3} to embeddings $X_\Delta \to X_\Omega$ in terms of the algebraic geometry of $Y_\Delta$ and $Y_\Omega$?
\end{question}

Finally, our proof of Theorem \ref{thm:app_2} for the Gromov width requires only a family of capacities that provide sharp obstructions for embeddings of the ball into convex toric domains, and an extension of these invariants to closed toric surfaces satisfying a set of axioms (see Proposition \ref{thm:axioms_of_calg}). It is interesting to ask if the proof of Theorem \ref{thm:app_2} can be ported to higher dimensions using another family of holomorphic curve based capacities, such as the $S^1$-equivariant symplectic homology capacities of Gutt-Hutchings \cite{gh} or the rational SFT capacities of Siegel \cite{ks}.

\noindent 

\subsection*{Outline} This concludes {\bf \S 1}, the introduction. The rest of the paper is organized as follows.

\vspace{4pt}

In {\bf \S 2}, we cover preliminaries in Seiberg-Witten theory (\ref{subsec:SW_invariants}) and embedded contact homology (\ref{subsec:ECH}). We then prove an important estimate of the ECH capacities of a star-shaped domain in terms of a minimum area over Seiberg-Witten non-zero classes. We should note that this is where the ``neck stretching'' part of the argument is made formal.

\vspace{4pt}

In {\bf \S 3}, we discuss the algebraic capacities in earnest (\S \ref{subsec:cons_alg_cap}). We then prove Theorem \ref{thm:main_intro} using the results of \S 2 and methods from algebraic geometry (\S \ref{subsec:ech_alg}).  

\vspace{4pt} 

In {\bf \S 4}, we discuss the applications to toric surfaces. We start with a review of toric surfaces (\ref{subsec:toric_surfaces}) and toric domains (\ref{subsec:toric_domains}). We then show that the algebraic capacities of a (possibly singular) surface satisfy a set of nice axioms (\ref{subsec:axioms_of_calg}). Finally, we apply the axioms to prove Theorems \ref{thm:app_1}-\ref{thm:app_3}.

\subsection*{Acknowledgements} We would like to thank Michael Hutchings for sharing \cite{mh} with us and suggesting that its contents were relevant to the arguments in \S \ref{subsec:intuitive_argument}. We would also like to thank David Eisenbud and Sam Payne for helpful conversations. JC was supported by the NSF Graduate Research Fellowship under Grant No.~1752814.

\section{ECH capacities and Seiberg--Witten theory} \label{sec:preliminaries} In this section, we review some aspects of Seiberg--Witten theory (\S \ref{subsec:SW_invariants}) and embedded contact homology (\S \ref{subsec:ECH}). Our goal is to prove an estimate for the ECH capacities in terms of the Seiberg--Witten invariants in \S \ref{subsec:ECH_to_SW}. 

\subsection{Seiberg--Witten invariants} \label{subsec:SW_invariants} The Seiberg--Witten invariant of a closed $4$-manifold $X$ with $b^+(X) \ge 1$ and a spin-c structure $\mathfrak{s}$ is an integral smooth invariant denoted by
\[\SW_X(\mathfrak{s}) \in \Z\]
A symplectic manifold $X$ has a canonical spin-c structure $\mathfrak{s}_X$. Since spin-c structures on $X$ are a torsor for $H_2(X;\Z)$, $\SW_X$ in the symplectic setting can be viewed as map
\begin{equation}
\SW_X:H_2(X;\Z) \to \Z \qquad A \mapsto \SW_X(A) := \SW_X(\mathfrak{s}_X + A)
\end{equation}
In later sections (e.g. \S \ref{subsec:ech_alg}), we will often refer to the set of mod 2 Seiberg-Witten non-zero classes
\begin{equation}
\SW(X) := \{A \in H^2(X;\Z) \; : \; \SW_X(A) = 1 \mod 2\} \subset H^2(X;\Z)
\end{equation}
In this section, we discusss several properties of the these invariants that we will apply in later sections. See \cite{jm,sm} for a more detailed review. 

\vspace{3pt}

Let us briefly recall the construction of $\SW_X$ for $X$ symplectic and spin-c structure $\mathfrak{s} = \mathfrak{s}_X + A$. Choose a metric $g$ and a self-dual $2$-form $\mu$. Given this data, we can consider the Seiberg-Witten equations for a pair $(a,\psi)$ of a spin-c connection $a$ on $\mathfrak{s}$ and a spinor $\psi \in \Gamma(S^+)$.
\begin{equation}
D_a\psi = 0 \qquad F_a^+ + \sigma(\psi) = \mu
\end{equation}
\begin{prop} For generic $(g,\mu)$, the moduli space $\mathcal{M}(A)$ of pairs $(a,\psi)$ modulo a natural $C^\infty(X;S^1)$ action is a closed manifold of dimension
\begin{equation}
\Ind(A) = c_1(X) \cdot A + A^2
\end{equation} \end{prop}
\noindent The Seiberg-Witten invariant $\SW_X(A)$ is acquired by integrating a certain natural top-dimensional cohomology class over $\mathcal{M}(A)$. It is independent of the choice of $(g,\mu)$ if $b^+(X) \ge 2$. 

\subsection{Wall--Crossing} \label{subsec:wall_crossing} When $b^+(X) = 1$, two different Seiberg-Witten invariants arise depending on the choice of $(g,\mu)$. More precisely, we have invariants
\begin{equation} \label{eqn:symplectic_chamber}
\SW_X(A) \quad\text{if $(g,\mu)$ satisfies}\quad \frac{i}{2\pi} \int_Y \mu \wedge \tau_g > [\omega] \cdot A\end{equation}
\begin{equation}
\SW_X^-(A) \quad\text{if $(g,\mu)$ satisfies}\quad \frac{i}{2\pi} \int_Y \mu \wedge \tau_g < [\omega] \cdot A
\end{equation}
Here $\tau_g$ is the unique self-dual, $g$-harmonic $2$-form satisfying $[\tau_g] = [\omega]$ in $H^2(Y)$, and if $g$ is compatible with $\omega$ then $\tau_g = \omega$. We refer to the space of data satisfying as (\ref{eqn:symplectic_chamber}) as the \emph{symplectic} chamber and the other space of data as the \emph{non-symplectic} chamber (cf. \cite[p. 463]{tb4}).

\vspace{3pt}

The two invariants $\SW_X$ and $\SW^-_X$ are related by a well-known wall-crossing formula. Here is a simple version of this formula that we will use momentarily  (cf. Li-Liu \cite[Prop 1.1]{ll}).

\begin{thm}[Wall Crossing] \label{thm:wall_crossing} Let $X$ be a closed symplectic $4$-manifold with $b_1(X) = 0$ and $b^+(X) = 1$, and let $A \in H_2(X)$ satisfy $I(A) \ge 0$. Then
\begin{equation}
\SW_X(A) = \SW^-_X(A) \pm 1\end{equation}
\end{thm}

\vspace{3pt}

\subsection{Gromov-Taubes} \label{subsec:Gromov_Taubes} There is a deep alternate formulation of the Seiberg-Witten invariants using $J$-holomorphic curves due primarily to Taubes \cite{tb3,tb4}, who introduced the \emph{Gromov-Taubes invariants}
\[
\Gr_X:H_2(X) \to \Z \qquad A \mapsto \Gr_X(A)
\]
Given a choice of compatible complex structure $J$ on $X$, $\Gr_X(A)$ is a signed count of points in a certain $0$-dimensional moduli space $\mathcal{M}_A(J)$ of disconnected $J$-curves $C$ in homology class $A$ that pass through $k = I(A)/2$ generic points of $X$. Of course, $J$ must be chosen so that $\mathcal{M}_A(J)$ is transversely cut out in an appropriate sense.

\begin{thm}[Taubes] \label{thm:SW_is_Gr} The Seiberg-Witten and Gromov-Taubes invariants agree, i.e. $\SW_X = \Gr_X$.
\end{thm}

Theorem \ref{thm:SW_is_Gr} is extremely powerful and has a number of surprising consequences. For example, we have the following effectiveness result.

\begin{prop}[Effective Classes] \label{prop:sw_in_ne} Let $Y$ be a smooth projective surface. Then every Seiberg-Witten non-zero class is effective, i.e. $\on{SW}(Y)\subseteq\neb(Y)$.
\end{prop}

\begin{proof} Let $J$ be the projective complex structure on $Y$. Any $J$-holomorphic map $u:\Sigma \to Y$ from a closed (possibly disconnected) Riemann surface $\Sigma$ is, of course, algebraic. If $A \in H_2(Y)$ is  non-zero and non-effective, then no such curve can exist. In particular, the Gromov-Taubes moduli space $\mathcal{M}_A(J)$ is empty (and thus transverse), so $\SW_Y(A) = \Gr_Y(A) = 0$.

\vspace{3pt}

There is a slight technical point when $A = 0$. In this case, the empty curve is counted as the unique $J$-curve of homology class $0$, so $\Gr_Y(0) = 1$. This covers the statement in that case.
\end{proof}

\subsection{SW For Rational Surfaces} \label{subsec:SW_for_rational_surfaces} We can use Proposition \ref{thm:wall_crossing} and the wall-crossing formula in Theorem \ref{thm:wall_crossing} to compute the mod 2 Seiberg-Witten invariants of a rational surface. This calculation is key to \S \ref{sec:alg_cap_bir}.

\begin{prop}[Rational Surfaces] \label{prop:rational_surfaces} Let $Y$ be a smooth rational surface. Then
\[
\SW(Y) = \{A \in \neb(Y) \; : \; I(A) \ge 0\}
\]
\end{prop}

The proof is a direct generalization of the calculation for $\P^2$, and requires the following lemma.

\begin{lemma} Every smooth, rational, projective surface $X$ admits a psc (positive scalar curvature) metric. \end{lemma}

\begin{proof} Every minimal rational surface $M$ has a psc metric \cite[Thm 1]{lb}, e.g. $\P^2$. In particular, $\bar{\P}^2$ also has a psc metric. Thus by the work of Gromov-Lawson \cite[Thm 1]{gl}, the connect sum $X = M \# k\bar{\P}^2$ has a psc metric for any $k \ge 0$. This covers all rational surfaces. \end{proof}

\begin{proof} (Proposition \ref{prop:rational_surfaces}) Let $Y$ be the smooth rational surface above with Kahler form $\omega$. Every rational projective surface $Y$ satisfies $b^+(Y) = 1$ and $b_1(Y) = 0$. 

\vspace{3pt}

By Proposition \ref{prop:sw_in_ne}, every Seiberg-Witten non-zero class $A$ is effective. Furthermore, every such class $A$ must satisfy $I(A)$ since the moduli space $\mathcal{M}(A)$ must have dimension $I(A) \ge 0$. Thus
\[
\SW(Y) \subseteq \{A \in \neb(Y) \; : \; I(A) \ge 0\}
\]
We thus must prove inclusion in the other direction.

\vspace{3pt}

Thus let $A \in H_2(Y)$ be an effective class with $I(A) \ge 0$. Let $(g,\mu)$ be a pair of a psc metric and a $C^0$-small self-dual $1$-form $\mu$. The pair is in the non-symplectic chamber. Indeed, $[\omega] \cdot A > 0$ since $\omega$ is ample and $A$ is effective, while $\int \mu \wedge \tau_g \simeq 0$. The $\mu$-perturbed Seiberg-Witten moduli space is empty since $g$ has psc \cite[Cor 2.2.6 and Cor 2.2.18]{ni}. Thus the Seiberg-Witten invariant $\SW^-(A)$ in this chamber vanishes. By the wall-crossing formula of Theorem \ref{thm:wall_crossing}, we thus conclude that
\[\SW^+_Y(A) = \SW^-_Y(A) \pm 1 = 1 \mod 2\]
This concludes the proof. \end{proof}

Here are a few examples of the above calculation for specific rational surfaces.

\begin{example}[Projective Plane] The homology $H_2(\P^2)$ is generated by the hyperplane class $H$ and the effective classes are $\NE(\P^2) = \text{Cone}(H)$. Furthermore, the anti-canonical is $-K = 3H$ so
\[
I(kH) = kH \cdot 3H + kH \cdot H = (k^2 + 3k)
\]
Thus $I(A) \ge 0$ for any effective class and so by Proposition \ref{prop:rational_surfaces}, $\SW(\P^2) = \text{Cone}(H)$. 
\end{example}

\begin{example}[Line Times Line] The effective cone $\NE(\P^1 \times \P^1)$ is generated by the two classes $D_1 = [\P^1 \times p]$ and $D_2 = [p \times \P^1]$. These intersect as follows.
\[
D_1 \cdot D_1 = D_2 \cdot D_2 = 0 \qquad D_1 \cdot D_2 = 1
\]
The anti-canonical divisor $-K$ is $2D_1 + 2D_2$. This is an ample class, so again $I(A) \ge 0$ for any $A$ and we acquire $\SW(\P^1 \times \P^1) = \text{Cone}(D_1,D_2)$.\end{example}

\subsection{Embedded Contact Homology} \label{subsec:ECH} Here we review embedded contact homology as a symplectic field theory, as presented in \cite{mh} (also see \cite{mh2}).

\begin{definition} A \emph{contact $3$-manifold} $(Y,\xi)$ is a $3$-manifold $Y$ with a $2$-plane bundle $\xi \subset TY$ that is the kernel $\xi = \ker(\alpha)$ of a contact form. A \emph{contact form} $\alpha$ is a $1$-form satisfying
\[\alpha \wedge d\alpha > 0 \quad\text{everywhere}\]
The \emph{Reeb vector-field} $R$ of $\alpha$ is the unique vector-field satisfying $\alpha(R) = 1$ and $d\alpha(R,\cdot) = 0$, and a \emph{Reeb orbit} is a closed orbit (modulo reparametrization) of $R$.
\end{definition}

The \emph{embedded contact homology}, or ECH for short, of a closed contact $3$-manifold $(Y,\xi)$ is a $\Z/2$-graded $\Z/2$-module denoted by
\[\ECH(Y,\xi)  = \bigoplus_{[\Gamma] \in H_1(Y;\Z)}\ECH(Y,\xi;[\Gamma])\]
The ECH group comes equipped with a degree $-2$ \emph{U-map} and a distinguished \emph{empty set} class.
\[U:\ECH(Y,\xi;[\Gamma]) \to \ECH(Y,\xi;[\Gamma]) \qquad [\emptyset] \in \ECH(Y,\xi;[0])\]
The $\Z/2$ grading on $\ECH(Y,\xi;[0])$ can be canonically enhanced to a $\Z/2m$-grading where $[\emptyset]$ has grading $0$ and $m$ is defined by
\[m := \min{\langle c_1(\xi);[\Sigma]\rangle \; : \; [\Sigma] \in H_2(Y;\Z)}\]
The simplest example of ECH groups are those of the 3-sphere.

\begin{prop} \label{prop:ECH_of_3_sphere} (c.f. \cite{mh2}) The embedded contact homology $\ECH(S^3,\xi)$ of the $3$-sphere is given by
\[\ECH(S^3,\xi) = \Z/2[U^{-1}]\]
as a $\Z/2[U]$-module, where $|U^{-1}| = 2$ and $U$ acts in the obvious way.
\end{prop}

Given a choice of contact form $\alpha$ for $(Y,\xi)$, one can enhance the ECH groups of $Y$ to a family of \emph{filtered ECH groups} $\ECH^L(Y,\alpha;[\Gamma])$ parametrized by $L \in [0,\infty)$ equipped with natural maps
\begin{equation} \label{eqn:filtered_ECH_inclusions} \iota_L^K:\ECH^L(Y,\alpha;[\Gamma]) \to \ECH^K(Y,\alpha;[\Gamma]) \quad\text{and}\quad \iota_L:\ECH^L(Y,\alpha;[\Gamma]) \to \ECH(Y,\xi;[\Gamma])\end{equation} 
Each filtered ECH group comes equipped with a U-map and empty set class, and these structures are compatible with the maps (\ref{eqn:filtered_ECH_inclusions}).
\[U^L:\ECH^L(Y,\alpha;[\Gamma]) \to \ECH^L(Y,\alpha;[\Gamma]) \qquad [\emptyset]^L \in \ECH^L(Y,\xi;[0])\]
Furthermore, the inclusions $\iota^K_L$ respect composition and the ordinary ECH is the colimit of the filtered ECH groups via the maps $\iota_L$. 

\vspace{3pt}

We can give a simple definition of the ECH capacities in terms of the formal structure of ECH described above.
\begin{definition} The \emph{k-th ECH capacity} $c_k(Y,\alpha)$ of a closed contact $3$-manifold is defined by
\[c_k(Y,\alpha) = \min{L \; :\; [\emptyset] = U^k \circ \iota_L(\sigma) \text{ for }\sigma \in \ECH^L(Y,\alpha;[0])}\]
The \emph{$k$-th ECH capacity} $c_k(X,\lambda)$ of a Liouville domain $(X,\lambda)$ is the $k$-th ECH capacity of its boundary $(\partial X,\lambda|_{\partial X})$ as a contact manifold.
\end{definition}

The ECH capacities are (non-normalized) capacities on the category of Liouville domains.

\begin{prop} \label{prop:ECH_capacities_axioms} The ECH capacities $c_k(\cdot)$ satisfy the following axioms.
\begin{itemize}
	\item[(a)] (Inclusion) If $X \to X'$ is a symplectic embedding of Liouville domains, then $c_k(X,\lambda) \le c_k(X',\lambda')$.
	\item[(b)] (Scaling) If $(X,\lambda)$ is a Liouville domain then $c_k(X,C \cdot\lambda) = C \cdot c_k(X,\lambda)$ for any constant $C > 0$.
\end{itemize}
\end{prop}

The ECH groups are the homology of an ECH chain group $\text{ECC}(Y,\alpha,J)$ depending on a choice of non-degenerate\footnote{A non-degenerate contact form is one where the differential of the Poincare return map along any orbit has no $1$-eigenvalues.} contact form $\alpha$ and a complex structure $J$ on the symplectization of $Y$ satisfying certain conditions. The chain group is freely generated over $\Z/2$ by \emph{orbit sets}
\[\Gamma = \{(\gamma_i,m_i)\}_{i=1}^k \qquad \gamma_i \text{ is an embedded Reeb orbit and }m_i \in \Z_+\]
satisfying the condition that $m_i = 1$ if the orbit $\gamma_i$ is hyperbolic. Given an element $x$ of $\text{ECC}(Y,\alpha,J)$ and an orbit set $\Gamma$, we denote the $\Gamma$-coefficient of $x$ by $\langle x,\Gamma\rangle$. The differential
\[\partial:\text{ECC}(Y,\alpha,J) \to \text{ECC}(Y,\alpha,J)\] is defined by a holomorphic curve count. More precisely, if $\Gamma_+ = \{(\gamma_i,m_i)\}_1^k$ and $\Gamma_- = \{(\eta_i,n_i)\}_1^l$ are admissible orbit sets, then the $\Gamma_-$-coefficient of $\partial \Gamma_+$ is given by
\[\langle \partial \Gamma_+,\Gamma_-\rangle = \# \mathcal{M}_1(Y,J)/\R\]
Here $\# \mathcal{M}_1(Y,J)/\R$ is a count of (possibly disconnected) holomorphic curves in the symplectization of $Y$ that have ECH index $1$ with positive ends at $\Gamma_+$ and negative ends at $\Gamma_-$. The \emph{ECH index} $I(C)$ of a homology class in $C \in H_2(Y,\Gamma_+ \cup \Gamma_-)$ is defined by 
\begin{equation} \label{eqn:ECH_index}
I(C) = \langle c_\tau(\xi),C\rangle + Q_\tau(C,C) + \sum_{i=1}^k \sum_{j=1}^{m_i} \CZ_\tau(\gamma^i_j) - \sum_{i=1}^l \sum_{j = 1}^{n_i} \CZ_\tau(\eta^i_j)
\end{equation}
Here $c_\tau(\xi)$ is the relative 1st Chern class, $Q_\tau(C,C)$ is the relative intersection form and $\CZ_\tau$ is the Conley-Zehnder index (all relative to a trivialization $\tau$ of the contact structure).

\vspace{4pt}

Embedded contact homology has a vaguely TQFT-like structure, whereby certain types of cobordisms between contact manifolds induce maps on the (filtered) ECH groups.

\begin{definition} A \emph{strong symplectic cobordism} $X$ between contact manifolds $Y_\pm$ with contact form, denoted by
\[(X,\omega):(Y_+,\alpha_+) \to (Y_-,\alpha_-)\]
is a symplectic manifold $(X,\omega)$ with oriented boundary $\partial X = Y_+ - Y_-$ such that $\omega|_{Y_\pm} = \pm d\alpha_\pm$. The \emph{area class} $[\omega,\alpha_\pm] \in H^2(X,\partial X)$ of $(X,\omega)$ is the class of the relative de Rham cycle
\[
(\omega, \alpha_+ - \alpha_-) \in \Omega^2(X) \oplus \Omega^1(Y_+ \cup -Y_-)
\]
\end{definition}

We use $[\Sigma]:[\Gamma_+] \to [\Gamma_-]$ to denote a relative class in $H_2(X,\partial X)$ whose image under the boundary map $\partial:H_2(X,\partial X) \to H_1(\partial X)$ is given by
\[
[\Gamma_+] \oplus [\Gamma_-] \in H_2(Y_+) \oplus H_2(Y_-) \simeq H(\partial X)
\]
For convenience, we use $\rho[\Sigma] \to \R$ denote the pairing of $[\Sigma]$ with the area class $[\omega,\alpha_\pm]$. Explicitly, we have the formula
\[\rho[\Sigma] = \int_{\Sigma} \omega - \int_{\partial_+\Sigma} \alpha_+ + \int_{\partial_-\Sigma} \alpha_-\]
With the above notation, we can state the following result of Hutchings regarding the functoriality of ECH with respect to strong symplectic cobordisms.

\begin{thm}[Hutchings \cite{mh}] \label{thm:ECH_as_SFT} A strong symplectic cobordism $X:Y_+ \to Y_-$ and let $[\Sigma]:[\Gamma_+] \to [\Gamma_-]$ be a class in $H_2(X,\partial X)$. Then there is a canonical, ungraded map
\begin{equation}
\ECH^L(X;[\Sigma]):\ECH^L(Y_+,\alpha_+;[\Gamma_+]) \to \ECH^{L + \rho[\Sigma]}(Y_-,\alpha_-;[\Gamma_-])
\end{equation}
\begin{equation}
\ECH(X;[\Sigma]):\ECH(Y_+,\xi_+;[\Gamma_+]) \to \ECH(Y_-,\xi_-;[\Gamma_-])
\end{equation}
These maps are compatible with composition, and satisfy the following axioms.
\begin{itemize}
	\item[(a)] (Curve Counting) There exists a chain map $\Phi$ inducing $\ECH(X;[\Sigma])$, of the form
     \[\Phi^L:\text{ECC}_L(Y_+,\alpha_+;[\Gamma_+]) \to \text{ECC}^{L + \rho[\Sigma]}(Y_-,\alpha_-;[\Gamma_-])\]
     that ``counts curves'' in the following sense: if $\Gamma_\pm$ are orbit sets in $Y_\pm$ such that
     \[\langle \Phi_A(\Gamma_+),\Gamma_-\rangle = 1\]
     then there is a holomorphic current\footnote{This is a formal positive integer combination of embedded holomorphic curves.} $C$ of ECH index $1$ asymptotic at $\pm \infty$ to $\Gamma_\pm$.
	\item[(b)] (Filtration) The maps commute with the inclusion maps $\iota^K_L$ and $\iota_L$, e.g.
	\[\begin{tikzcd} 
     \ECH^L(Y_+,\alpha_+;[\Gamma_+]) \arrow{rr}{\ECH^L(X;[\Sigma])} \arrow{d}{} & \; & \ECH^{L + \rho[\Sigma]}(Y_-,\alpha_-;[\Gamma_-]) \arrow{d}{}\\
     \ECH^K(Y_+,\alpha_+;[\Gamma_+]) \arrow{rr}{\ECH^K(X;[\Sigma])}  & \; & \ECH^{K + \rho[\Sigma]}(Y_-,\alpha_-;[\Gamma_-]) \end{tikzcd}\]
	\item[(c)] (U-Map) The maps commute with the $U$-maps, e.g.
	\[U^{L + \rho[\Sigma]} \circ \ECH^L(X;[\Sigma]) = \ECH^L(X;[\Sigma]) \circ U^L\]
%	\item[(c)] (Taubes Isomorphism) The maps commute with the Taubes isomorphism.
%	\[\begin{tikzcd} 
%     \ECH(Y_+,\xi_+;[\Gamma_+]) \arrow{rr}{\ECH(X;[\Sigma])} \arrow{d}{\text{Tb}} & \; & \ECH(Y_-,\xi_-;[\Gamma_-]) \arrow{d}{\text{Tb}}\\
%     \HM(Y_+,\mathfrak{s}_+ + [\Gamma_+]) \arrow{rr}{\HM(X,\mathfrak{s}_\omega + [\Sigma])}  & \; & \ECH(Y_-,\mathfrak{s}_- + [\Gamma_-]) \end{tikzcd}\]
     \item[(d)] (Seiberg--Witten) Let $(P,\xi)$ be a contact $3$-manifold. Consider a pair of strong symplectic cobordisms and their composition, denoted by
     \[N:\emptyset \to P \qquad X:P \to \emptyset \quad\text{and}\quad Y = N \cup_Z X:\emptyset \to \emptyset\] 
     Fix homology classes $[A] \in H_2(N,Z)$ and $[B] \in H_2(X,Z)$ with $\partial[A] = \partial[B]$. Then
     \[\ECH(X,[B]) \circ U^k \circ \ECH(N,[A]) = \sum_{[C] \in S} \SW_Y([C])\]
     Here $S \subset H_2(X)$ is shorthand for the set of homology classes satisfying
     \[[C] \cap N = [A] \qquad [C] \cap X = [B] \quad\text{and}\quad I([C]) = 2k\]
\end{itemize}
\end{thm}

The analogue functoriality result for exact symplectic cobordisms is well established \cite[Theorem 1.9]{ht2013} and has been used extensively, e.g. to define the ECH capacities \cite{mh4}. Non-exact cobordisms and the foundations provided by \cite{mh} have been used to define Gromov-Taubes invariants for non-symplectic manifolds \cite{g1,g2}. 

\begin{remark}[Proof of Theorem \ref{thm:ECH_as_SFT}] Since a detailed treatment of Theorem \ref{thm:ECH_as_SFT} has yet to appear in the literature outside of \cite{mh}, we include a brief discussion of its proof. It is similar to the exact case in \cite{ht2013} with some small modifications.

 The basic strategy of \cite{ht2013} and \cite{mh} is to establish a filtered version of the Taubes isomorphism between filtered ECH and an energy filtered version of Kronheimer-Mrowka's monopole Floer homology (MFH) groups \cite[\S 3]{ht2013}. Cobordism maps on filtered ECH can then be \emph{defined} so that they intertwine the analogous maps on filtered MFH via these isomorphisms \cite[\S 5.1]{ht2013}. Theorem \ref{thm:ECH_as_SFT}(c)-(d) follow more or less immediately from this strategy \cite[Cor 5.3]{ht2013}.

The proof in \cite[\S 6]{ht2013} of the analogue of Theorem \ref{thm:ECH_as_SFT}(a)-(b) uses a well-known argument for producing instantons counted in MFH cobordism maps from holomorphic curves \cite[\S 6.2]{ht2013} and an SFT/Gromov compactness argument \cite[\S 6.4]{ht2013}. In the non-exact setting, the required compactness can be guaranteed by only considering cobordism maps $\on{MFH}(X,\mathfrak{s})$ in MFH induced by a symplectic cobordism $(X,\omega)$ equipped with a specific spin-c structure $\mathfrak{s}$, determined by a fixed relative homology class $[\Sigma]:[\Gamma_+] \to [\Gamma_-]$ with $\mathfrak{s} = \mathfrak{s}_\omega + PD[\Sigma]$. The energy of the instantons and holomorphic curves involved in $\on{MFH}(X,\mathfrak{s})$ obey a uniform bound in terms of the actions of the ends $\Gamma_\pm$ and $\rho[\Sigma]$. In particular, the moduli space of curves admits a compactification in the SFT topology and the arguments of \cite[\S 6]{ht2013} can be slightly modified to handle this case. \end{remark}

\subsection{From ECH to SW} \label{subsec:ECH_to_SW} We now conclude the section by applying the formal structure of the ECH groups in \S \ref{subsec:ECH} to estimate for the ECH capacities of a star-shaped domain embedded into closed symplectic manifolds.

\begin{prop} \label{prop:embedding_ECH_bound} Let $(X,\lambda) \subset \R^4$ be a star-shaped domain with restricted Liouville form $\lambda$ and let $(Y,\omega)$ be a closed symplectic $4$-manifold. Fix an embedding
\[\iota:(X,d\lambda) \to (Y,\omega)\]
Then the ECH capacities of $X$ satisfy
\begin{equation} \label{eqn:embedding_ECH_bound}
c_k(X) \le \inf_{[\Sigma] \in \SW(Y)}\{\langle \omega,[\Sigma]\rangle \; : \; I([\Sigma]) \ge 2k \}
\end{equation}
\end{prop}

\begin{remark} This result is based on the proofs in \cite[\S 2.2]{mh}. \end{remark}

\begin{proof} Let $(Z,\alpha)$ be the contact boundary of $(X,\lambda)$ and let $[\Sigma] \in H_2(Y)$ be any $\Z$-homology class satisfying the constraints laid out in (\ref{eqn:embedding_ECH_bound}).
\[ \SW_Y([\Sigma]) = 1 \mod 2 \; \text{ and } \; I([\Sigma]) \ge 2k\]
It suffices to demonstrate the following inequality for any such $[\Sigma]$.
\[c_k(X) \le A := \langle \omega,[\Sigma]\rangle\]
Since $c_k(X) \le c_j(X)$ for $j = I([\Sigma])/2$, we can assume that $k = j = I([\Sigma])/2$. Furthermore, it is equivalent to show that for all $\epsilon > 0$ sufficiently small, there exists a class
\begin{equation} \label{eqn:eta_formula} \eta \in \ECH^{A + \epsilon}(Z,\xi;[0]) \quad\text{with}\quad U^k\iota_A\eta = [\emptyset] \in \ECH(Z,\xi;[0])\end{equation}

To find an $\eta$ that satisfies (\ref{eqn:eta_formula}), we consider the splitting of $Y$ into $X$ (or rather, the image $\iota(X)$) and $N = Y \setminus X$. If we denote the contact boundary of $X$ by $(Z,\xi)$, we can interpret this as pair of strong symplectic cobordisms
\[N:\emptyset \to Z \qquad X:Z \to \emptyset\]
Since $X$ is diffeomorphic to a $4$-ball, the pair of maps
\[
H_2(Y) \xrightarrow{- \cap X} H_2(X,\partial X) \quad\text{and}\quad H_2(Y) \xrightarrow{- \cap P} H_2(P,\partial P)
\]
are, respectively, the $0$ map and an isomorphism. Let $[S] = [\Sigma] \cap X$ be the intersection of $[\Sigma]$ with $X$. Note that we have
\[
A = \langle [\omega],[\Sigma]\rangle = \rho[S] + \rho[0] = \rho[S] \]
Now we let $\epsilon > 0$ be small and arbitrary, and define the desired class $\eta$ by
\[
\eta = \ECH^A(P;[S])[\emptyset] \in \ECH^{A + \epsilon}(Z,\xi;[0]) \quad\text{where}\quad [\emptyset] \in \ECH^\epsilon(\emptyset;[0]) \simeq \Z/2[\emptyset]
\]

We would like to show that $U^k \iota_{A + \epsilon} \eta = [\emptyset]$. To start, pick a chain map lifting the ECH cobordism map as in Thm. \ref{thm:ECH_as_SFT}(a). That is,
\[
\Phi:\Z/2 \to \text{ECC}^{\epsilon + A}(Z,\alpha;[0]) \quad \text{with} \quad [\Phi(\emptyset)] = \eta
\]
If $\Gamma_-$ is any orbit set such that $\langle \Phi(\emptyset),\Gamma_-\rangle = 1$, then by Theorem \ref{thm:ECH_as_SFT}(a) we know that there is a holomorphic current $C$ of ECH index $0$ with empty positive boundary and negative boundary $\Gamma_-$. If we let $C' \subset Z$ be a surface with positive boundary $\Gamma_-$, so that $|\Gamma_-| = I(C')$, then by the additivity of the ECH index we have
\[2k =I([\Sigma]) = I(C) + I(C') = I(C') = |\Gamma_-|\]
Thus we know that $\eta$ is homogenous of grading $2k$, and so $U^k \circ \iota_{A + \epsilon}(\eta)$ is grading $0$. In particular, by Proposition \ref{prop:ECH_of_3_sphere}, we have
\[U^k\iota_{A + \epsilon}\eta \in ECH_0(Z,\xi;[0]) \simeq ECH_0(S^3;[0]) = \Z/2[\emptyset]\]
On the other hand, by Theorem \ref{thm:ECH_as_SFT}(b) and (d), we know that
\[\ECH(X;[0]) \circ U^k \circ \iota_{A + \epsilon} \eta = \ECH^A(X;[0]) \circ U^k \circ \ECH(X;[S]) [\emptyset] = c\]
Here $c \in \Z/2$ is the sum over $[C]$ with $[C] \cap X = [0]$ and $[C] \cap P = [S]$ of $\SW_Y([C]) \mod 2$. Since $[\Sigma]$ is the unique such class and $\SW_Y([\Sigma]) = 1 \mod 2$, we find that $c = 1$. Thus, $U^k\iota_{A + \epsilon}\eta$ is non-zero and we must have
\[
U^k\iota_{A + \epsilon}\eta = [\emptyset]
\]
This proves that for every $\epsilon$, there is a class $\eta \in ECH^{A + \epsilon}(Z,\xi;[0])$ satisfying (\ref{eqn:eta_formula}), and thus concludes the proof.\end{proof}

\begin{remark} The proof of Proposition \ref{prop:embedding_ECH_bound} generalizes immediately to Liouville domains $(X,\lambda)$ that satisfy the following criteria.
\begin{itemize}
	\item[(a)] The map $H_2(\partial X) \xrightarrow{\iota_*} H_2(X)$ is $0$.
	\item[(b)] The contact manifold $(\partial X,\xi)$ has torsion chern class, i.e. $c_1(\xi) = 0 \in H^2(\partial X;\Q)$.
	\item[(c)] The empty set $[\emptyset]$ is the unique class of ECH grading $0$ in the image of the $U$-map.
\end{itemize}
The conclusion of Proposition \ref{prop:embedding_ECH_bound} must be appropriately modified so that (\ref{eqn:embedding_ECH_bound}) is a minimum over all classes $[\Sigma]$ such that $[\Sigma] \cap X = 0$. In practice, the most difficult criterion to verify is (c). This holds, for instance, when $[\emptyset]$ is the unique ECH index $0$ class. It is also believed to hold for circle bundles over a $2$-sphere (c.f.~the unpublished thesis of Ferris \cite{mf} and the forthcoming work of Nelson--Weiler \cite{nw}).
\end{remark}

\section{Algebraic capacities and birational geometry} \label{sec:alg_cap_bir} We now construct of the algebraic capacities (\S \ref{subsec:cons_alg_cap}) and prove Theorem \ref{thm:main} (\S \ref{subsec:ech_alg}). 

\begin{conventions} In this section, all surfaces will be projective normal algebraic surfaces over the complex numbers, not necessarily smooth, unless otherwise specified. \end{conventions}

Let $\bK\in\{\Z,\Q,\R\}$. We work in the N\'eron--Severi group $\on{NS}(Y)\subseteq H^2(Y,\Z)$ of Weil $\Z$-divisors regarded up to algebraic equivalence. We denote $\on{NS}(Y)_\bK:=\on{NS}(Y)\otimes_\Z\bK$. We say that a $\Z$-divisor $D$ on a surface $Y$ is $\Q$\emph{-Cartier} if some integer multiple of $D$ is Cartier; that is, the sheaf $\mO(D)$ is a line bundle. $Y$ is said to be $\Q$\emph{-factorial} if every Weil $\Z$-divisor on $Y$ is $\Q$-Cartier. Every toric surface is $\Q$-factorial. A $\Q$-Cartier $\R$-divisor $D$ on $Y$ is \emph{nef} if $D\cdot C\geq0$ for all curves $C\subseteq Y$. Denote by $\on{Nef}(Y)_\bK$ the classes in $\on{NS}(Y)_\bK$ corresponding to nef divisors.

\subsection{Construction of algebraic capacities} \label{subsec:cons_alg_cap}

Let $Y$ be a $\Q$-factorial projective surface and let $A$ be an ample $\R$-divisor on $Y$. We recall the optimisation problems of \cite{bwo,bwr,bwt} that are designed to emulate ECH capacities in the context of algebraic geometry.

\begin{definition}[{\cite[\S4.5]{bwo} or \cite[Def.~2.2]{bwr}}] \label{def:alg_cap} The $k$\emph{th algebraic capacity} of $(Y,A)$ are given by
\begin{equation} \label{eqn:alg_cap_def} \calg_k(Y,A):=\inf_{D\in\on{Nef}(Y)_\Z}\{D\cdot A:\chi(D)\geq k+\chi(\mO_Y)\}\end{equation}
\end{definition}

\begin{remark} Note that it follows from Kleiman's criterion for nef-ness that this infimum in (\ref{eqn:alg_cap_def}) is always achieved. \end{remark}

The \emph{index} of a $\Z$-divisor $D$ on $Y$ is given by $I(D):=D\cdot (D-K_Y)$. When $Y$ is smooth or has at worst canonical singularities \cite{ypg} we have Noether's formula
\begin{equation} \label{eqn:noether} \chi(D)=\chi(\mO_Y)+\frac{1}{2}I(D)\end{equation}
Furthermore, if $\omega_A$ is the Kahler class induced by $A$ via the embedding into $\pr H^0(k\mO(A))$ for $k \gg 0$ (which is defined because $A$ is ample) we may write
\begin{equation} \label{eqn:area_vs_intersection} D\cdot A = \langle \omega_A,D\rangle = \int_D\omega_A\end{equation}
In these cases, we can alternatively write the algebraic capacities as
\begin{equation} \label{eqn:calg_intermiediate}\calg_k(Y,A)=\inf_{D\in\on{Nef}(Y)_\Z}\{\langle \omega_A,D\rangle \; : \; I(D)\geq 2k\}\end{equation}
which is very similar to the upper bound for $\cech_k$ in Proposition \ref{prop:embedding_ECH_bound}.

\subsection{Relating ECH capacities and algebraic capacities} \label{subsec:ech_alg}

We seek to prove the following result.

\begin{thm}\label{thm:sw_nef} Suppose $Y$ is a smooth rational surface, and let $A$ be an ample $\R$-divisor on $Y$. Then
$$\inf_{D\in\on{SW}(Y)}\{D\cdot A:I(D)\geq 2k\}=\inf_{D\in\on{Nef}(Y)_\Z}\{D\cdot A:I(D)\geq 2k\}=:\calg_k(Y,A)$$
\end{thm}

By combining Proposition  \ref{prop:embedding_ECH_bound}, the formula (\ref{eqn:calg_intermiediate}) and Theorem \ref{thm:sw_nef}, we immediately acquire the main result, which we state again for completeness.

\begin{thm}\label{thm:main} Suppose that $X \to Y$ is a symplectic embedding of a star-shaped domain $X$ into a smooth rational projective surface $Y$ with a ample $\R$-divisor $A$ and symplectic form $\omega_A$ satisfying $[\omega_A] = \text{PD}[A]$. Then
\begin{equation} \label{eqn:main} \tag{$\star$}
\cech_k(X)\leq\calg_k(Y,A)\end{equation}
\end{thm}

\begin{remark} We only require an upper bound of the Seiberg-Witten quantity by $\calg_k$ for the purposes of this paper. However, Theorem \ref{thm:sw_nef} is satisfying because it demonstrates that the algebraic capacities are (as obstructions) just as sensitive as the Seiberg-Witten theoretic quantities. \end{remark}

We treat the case of smooth rational surfaces using the Minimal Model Program. To begin, recall that the nef cone is contained within the effective cone, i.e.
\[\on{Nef}(Y)_{\Z} \subseteq \neb(Y)\]
We calculated the Seiberg-Witten theory of a rational surface in Proposition \ref{prop:rational_surfaces}. That calculation implies the inequality of Theorem \ref{thm:sw_nef}
\[
\inf_{D\in\on{SW}(Y)}\{D\cdot A:I(D)\geq 2k\} = \inf_{D\in\neb(Y)}\{D\cdot A:I(D)\geq 2k\}
\]
This immediately implies that we have an inequality in one direction.
\begin{equation} \label{eqn:inf_SW_vs_NEF} \inf_{D\in\on{SW}(Y)}\{D\cdot A:I(D)\geq 2k\}\leq\inf_{D\in\on{Nef}(Y)_{\Z}}\{D\cdot A:I(D)\geq 2k\}\end{equation}
For the converse inequality, we will show for that each Seiberg--Witten nonzero divisor there is a nef divisor that is `preferable' from the perspective of the optimisation problems above. For this purpose, we adopt the following terminology.

\begin{definition} \label{def:pref} Let $Y$ be a $\Q$-factorial surface. We say that a Weil $\Q$-divisor $D_0$ is 
\begin{itemize}
\item[(a)] \emph{index-preferable} to another Weil $\Q$-divisor $D$ if $I(D_0)\geq I(D)$ and
\item[(b)] \emph{area-preferable} $D$ if $D_0\cdot A\leq D\cdot A$ for all ample $\R$-divisors $A$ on $Y$.
\end{itemize}
A Weil $\Q$-divisor $D_0$ that is both area- and index-preferable will simply be called \emph{preferable}. Note that $D_0$ is area-preferable to $D$ if and only if $D-D_0$ is effective.
\end{definition}

To construct preferable divisors in general we will use the \emph{isoparametric transform} $\ip_Y$ of \cite{phys}. This takes an effective divisor $D$ to a new divisor $\ip_Y(D)$ given by
\begin{equation} \label{eqn:ip_def} \ip_Y(D) := D-\sum_{D\cdot D_i<0}\lc\frac{D\cdot D_i}{D_i^2}\rc D_i\end{equation}
Here the sum is over prime divisors $D_i$ with $D\cdot D_i<0$ and, in particular, $\ip_Y(D) = D$ if $D$ is nef. We denote by $\ip^n_Y(D)$ the result of iterating $\ip^n_Y$ $n$ times. In \cite{phys}, the following result is proven.

\begin{thm}[{\cite[Thm.~1.1 + 1.2]{phys}}] \label{thm:phys} For any effective divisor $D$ on a smooth surface $Y$ we have
$$h^0(D)=h^0(\ip_Y(D))$$
Then for all sufficiently large $n \gg 0$, we have $\ip^n_Y(D) = \ip^\infty_Y(D)$ for some nef $\ip^\infty_Y(D) \in \on{Nef}(Y)_\Z$. 
\end{thm}

We will need to know what $\ip_Y$ does to area and index. For area, the answer is quite simple.

\begin{lemma} \label{lem:iso_area_pref} Let $D$ be effective and $A$ be ample. Then $A \cdot \ip_Y(D) \le A \cdot D$.
\end{lemma}

\begin{proof} If $D_i$ is a prime divisor with $D_i \cdot D < 0$ and $D$ is effective, then $D_i^2 < 0$. Thus the coefficients of the sum in (\ref{eqn:ip_def}) are positive. Since $A$ is ample, $A \cdot D_i < 0$. These two facts imply the result.
\end{proof}

The answer for the index is more complicated. For this, we need the following lemma.

\begin{lemma} \label{lem:iso_index_pref} Let $Y$ be a smooth surface with $D$ an effective divisor on $Y$. Suppose $C_1,\dots,C_n$ is a collection of curves intersecting $D$ negatively. Then either one of the $C_i$ is a $(-1)$-curve or
$$I(D')\geq I(D) \quad\text{where}\quad D'=D-\sum_{i=1}^n\lc\frac{D\cdot C_i}{C_i^2}\rc C_i$$
In particular, $I(\ip_Y(D))\geq I(D)$ if no $(-1)$-curve intersects $D$ negatively.
\end{lemma}

\begin{proof} Suppose $n=1$ so that there is only one curve $C$. If $C^2=-1$ we are done, so let $C^2=-r$ for $r\geq2$. Let $D\cdot C=-\ell$ so that
$$D'=D-\lc\frac{\ell}{r}\rc C=:D-mC$$
Let $\pi\colon Y\to\ol{Y}$ be the contraction of $C$ to the singular surface $\ol{Y}$. We can compute
\begin{align*}
I(D')&=(D-mC)\cdot(D-mC-K_Y)\\
&=I(D)-2mD\cdot C+(-mC)\cdot(-mC-K_Y) \\
&=I(D)+2m\ell+(-mC)\cdot(-mC-\pi^*K_{\ol{Y}}-\frac{2-r}{r}C) \\
&=I(D)+2m\ell-m^2r-(2-r)m
\end{align*}
Now observe that $1>m-\frac{\ell}{r}\geq 0$ by definition and so $\ell+r>rm$. Furthermore, $r \ge 2$ and $m \ge 1$. Using these facts, we can compute the following lower bound.
\[2m\ell-mr(m+\frac{2-r}{r}) > 2m\ell-(\ell+r)(m+\frac{2-r}{r}) \]
\[= m\ell+\ell\cdot\frac{r-2}{r}-mr+r-2 \geq m\ell-r(m-1)-2  > (m-1)\ell-2 \geq -2\]
In particular, $I(D') > I(D) - 2$. However $I(\cdot)$ is even and so we must have $I(D')\geq I(D)$.

Now induct on the number of curves. Suppose the formula holds for a set of $n$ curves meeting an effective divisor negatively. Suppose curves $C_1,\dots,C_n,C$ intersect $D$ negatively. If any of the curves is a $(-1)$-curve then we are done. Assume not. Notate
$$D\cdot C=-\ell,\;\;\; C^2=-r,\;\;\;\lc\frac{D\cdot C}{C^2}\rc=m \;\;\text{and}\;\; F=\sum_{i=1}^{n-1}m_iC_i$$
so that $D'=D-F-mC$. Compute
\begin{align*}
I(D-F-mC)&= \\
&=I(D-F)+2mF\cdot C-2mD\cdot C+I(-mC) \\
&\geq I(D-F)+2m\ell-mr(m+\frac{2-r}{r}) \\
&>I(D-F)+(m-1)\ell-2 \\
&\geq I(D-F)-2
\end{align*}
where we used that $F\cdot C\geq0$ since $F$ is effective and supported away from $C$. By inductive assumption $I(D-F)\geq I(D)$ and so we have $I(D')>I(D)-2$. Since $I(\cdot)$ is even we can conclude that $I(D')\geq I(D)$ as desired.
\end{proof}

\begin{proof}[Proof of Thm.~\ref{thm:sw_nef}] We simply need to show that for any divisor in $\SW(Y)$, there exists a preferable nef divisor. In other words, we must construct a map
\[\cN_Y\colon\on{SW}(Y)\to\on{Nef}(Y)_\Z\]
taking a Seiberg--Witten nonzero divisor to a preferable nef $\Z$-divisor. We now construct these maps by induction on the number of blow ups necessary to make $Y$ from a minimal surface. 

\vspace{4pt}

For minimal rational surfaces the existence of an $\cN_Y$ is clear. In the cases of $\pr^2$ and $\pr^1\times\pr^1$, we have $\on{SW}(Y)=\on{Nef}(Y)_\Z$. Hirzebruch surfaces, on the other hand, have no $(-1)$-curves. Thus we can set $\cN_Y(D) = \ip_Y^n(D)$ for $n \gg 0$. Lemmas \ref{lem:iso_index_pref} and \ref{lem:iso_area_pref} imply that the result is preferable.

\vspace{4pt}

Now assume that such a function exists for all rational surfaces expressible as $b-1$ blowups of a minimal rational surface. Let $Y$ be a surface expressed as $b$ blowups of a minimal rational surface, and for any $(-1)$-curve $E\subseteq Y$ denote the contraction by $\pi_E\colon Y\to\ol{Y}_E$. We define $\cN_Y(D)$ by the following procedure.
\begin{enumerate}
\item[(a)] If $D\cdot C\geq0$ for all curves $C\subseteq Y$ then $D$ is nef and we define $\cN_Y(D)=D$.
\item[(b)] If $D\cdot E\leq 0$ for some $(-1)$-curve $E$, write $D=\pi_E^*\ol{D}+mE$ for some $\ol{D}\in\on{SW}(\ol{Y}_E)$ and for some $m\geq0$. The inductive hypothesis implies that there exists a nef $\Z$-divisor $\ol{D}_0$ preferable to $\ol{D}$. Define $\cN(D)=\pi_E^*\ol{D}_0$.
\item[(c)] If $D\cdot E>0$ for all $(-1)$-curves $E$ on $Y$ but $D\cdot C<0$ for some $(-r)$-curve $C$ with $r\geq 2$, recursively apply (a)-(c) to $\ip_Y(D)$ instead of $D$ and define $\cN_Y(D)$ as the result.
\end{enumerate}
This procedure terminates: if $\ip_Y^n(D)$ eventually intersects a $(-1)$-curve negatively then (b) outputs a nef divisor. If $\ip_Y^n(D)$ does not intersect a $(-1)$-curve nonpositively for any $n$ then after a finite number of steps we reach $\ip^\infty_Y(D)\in\on{Nef}(Y)_\Z$ by Theorem \ref{thm:phys}, which is returned by (a). Note that the application of Theorem \ref{thm:phys} is valid by Proposition \ref{prop:sw_in_ne}.

\vspace{4pt}

We claim that $\cN_Y(D)$ is nef and preferable to $D$. Indeed, all three steps (a)-(c) only improve the area and index constraints. This claim is trivial for (a) and follows from Lemmas \ref{lem:iso_area_pref} and \ref{lem:iso_index_pref} for (c). (b) produces a preferable nef $\Z$-divisor since $\pi_E^*\ol{D}$ is preferable to $D=\pi_E^*\ol{D}+mE$ from direct calculation (noting that $m\geq0$), and then $\pi_E^*\ol{D}_0$ is nef and preferable to $\pi_E^*\ol{D}$ since $\ol{D}_0$ is preferable to $\ol{D}$. \end{proof}

\section{Toric Surfaces} \label{sec:toric_surfaces} We now apply Theorem \ref{thm:main} to the study of embeddings into projective toric surfaces. We begin with a review of toric surfaces (\S \ref{subsec:toric_surfaces}) and toric domains (\S \ref{subsec:toric_domains}). We then demonstrate that the algebraic capacities on toric surfaces are uniquely characterized by a set of axioms (\S \ref{subsec:axioms_of_calg}). Finally, we discuss the main applications: obstructing embeddings of concave toric domains into toric surfaces, and monotonicity of the Gromov width under inclusion of moment polygons (\S \ref{subsec:emb_to_toric}).

\subsection{Toric varieties} \label{subsec:toric_surfaces} We start with a brief review of toric varieties. Our main reference is \cite{cls}.

\begin{definition} A \emph{(projective normal) toric variety} $Y$ of dimension $n$ over $\C$ is a projective normal variety with a $(\C^\times)^n$-action acting faithfully and transitively on a Zarisiki open subset of $Y$.
\end{definition}

Every toric variety $Y$ can be described (uniquely, up to isomorphism) by either a \emph{fan} $\Sigma \subset \R^n$ \cite[Def 3.1.2 and Cor 3.1.8]{cls} or a \emph{moment polytope} $\Omega \subset \R^n$ \cite[Def 2.3.14]{cls}. A fan $\Sigma$ for $Y$ can be recovered from a moment polytope $\Omega$ for $Y$ by passing to the \emph{inner normal fan} $\Sigma(\Omega)$ of $\Omega$ \cite[Prop 3.1.6]{cls}. We will focus on the polytope perspective, since it will be more important in this paper.

\begin{definition} \label{def:moment_polytope} A \emph{moment polytope} $\Omega \subset \R^n$ is a convex polytope with rational vertices and open interior. We denote the corresponding toric variety by $Y_\Omega$.
\end{definition}
\noindent Note that given a scalar $S > 0$ in $\Q$ or an affine map $T:\Z^2 \to \Z^2$, we can scale $\Omega$ to $S\Omega$ or apply $T$ to acquire $T\Omega$. There are naturally induced isomorphisms of varieties $Y_{S\Omega} \simeq Y_\Omega$ and $Y_{T\Omega} \simeq Y_\Omega$. 

\begin{definition} A \emph{smooth vertex} $v \in \Omega$ of a moment polytope is a vertex such that there exists a neighborhood  $U \subset \R_{\ge 0}^n$ of $0$, a neighborhood $V \subset \Omega$ of $v$, a scaling $S$ and a $\Z$-affine isomorphism $T$ such that $ST(U) = V$ and $ST(v) = 0$. Otherwise a vertex is \emph{singular}.
\end{definition}

\vspace{4pt}

On a projective toric variety, each face $F \subset \Omega$ determines a $\Q$-Cartier divisor $D_F$. Every torus invariant divisor is in the span of these divisors $D_F$, and every divisor class is represented by a torus-invariant divisor \cite[4.1.3]{cls}. Furthermore, every moment polytope $\Omega$ for a toric variety $Y_\Omega$ is associated to a natural divisor $A_\Omega$ given as a combination of these face divisors.

\begin{definition} \label{def:associated_ample} The \emph{associated divisor} $A_\Omega$ of the moment polytope $\Omega$ is defined as
\[A_\Omega = \sum_F a_F D_F\]
Here for each face $F \subset \Omega$, we define $u_F \in \Z^n$ and $a_F \in \Q$ by the following conditions.
\[\langle u_F,x\rangle = -a_F \text{ for }x \in F \qquad u_F \text{ is primitive in $\Z^n$, inward to $\Omega$ and normal to $F$}\]
Note that the equation $\langle u_F,x\rangle=-a_F$ defines a hyperplane that we denote by $\Pi_F$.
\end{definition}

\begin{lemma} The associated divisor $A_\Omega$ of a moment polytope $\Omega$ has the following properties.
\begin{itemize}
	\item[(a)] (Ample) $A_\Omega$ is an ample divisor, and so defines an projective embedding to projective space.
	\begin{equation} \label{eqn:associated_ample_embedding} |kD_\Omega|:Y_\Omega\to\pr H^0(Y_\Omega,kA_\Omega) \quad \text{ for }k \gg 0\end{equation}
	\item[(b)] (Translation/Scaling) Let $T \in \GL_n(\Z)$, $V \in \Z^n$ and $S \in \Q$. Then
	\[D_{T\Omega} = D_{\Omega} \qquad D_{\Omega + V} = D_\Omega + P_V \qquad D_{S\Omega} = S \cdot D_\Omega\]
	Here $P_V$ is a principle divisor depending on $V$.
\end{itemize}
\end{lemma}

\begin{proof} For (a), see \cite[Prop 6.1.10]{cls}. For (b), see \cite[\S 4.2, Ex 4.2.5(a)]{cls} for the translation property. The scaling and linear map properties follow from Definition \ref{def:associated_ample}. \end{proof}

More generally, any $\mathbb{T}^n$-equivariant $\Q$-divisor $D=\sum a_FD_F$ is associated to a half-space arrangement consisting of the half-spaces $H_F$ and a (possibly empty) polytope $P_F$ given by
\[H_F =\{x \in \R:\langle u_F,x\rangle\geq-a_F\} \qquad P_F = \cap_F \Pi_F\]
The dimension of the space of sections $h^0(D)$ is given by the number of lattice points $|P_F \cap \Z^n|$ \cite[\S 7.1, p. 322]{cls}. A divisor is ample if and only if $\partial H_F \cap P_D$ is an open subset of $\partial H_F$ for each $F$, and nef if $\partial H_F \cap P_D$ is non-empty for each $F$.

\vspace{4pt}

We are primarily interested in \emph{toric surfaces} , i.e. projective toric varieties of complex dimension $2$. In this case, the embedding (\ref{eqn:associated_ample_embedding}) gives $Y$ the structure of a symplectic orbifold by restriction of the Kahler form on $\P^N$. Every toric surface is an orbifold \cite[Thm.~3.1.19]{cls} since every two-dimensional fan is simplicial (dually, every polygon is simple).

\subsection{Toric domains} \label{subsec:toric_domains} We next review the theory of toric domains. Let $\omega_{\text{std}}$ denote the standard symplectic form on $\C^n$ and let $\mu$ denote the standard moment map
\[\mu:\C^n \to \R^n_{\ge 0} \qquad (z_1,\dots,z_n) \mapsto (\pi|z_1|^2,\dots,\pi|z_n|^2)\]

\begin{definition} A \emph{toric domain} $(X_\Omega,\omega)$ is the inverse image $\mu^{-1}(\Omega)$ of a closed subset $\Omega \subset [0,\infty)^2$ with open interior, equipped with the symplectic form $\omega_{\text{std}}|_{X_\Omega}$ and moment map $\mu|_{X_\Omega}$.

\vspace{4pt}

A toric domain $X_\Omega$ is \emph{convex} if $\Omega = C \cap [0,\infty)^n$ where $C \subset \C^n$ is a convex and contains $0$ in its interior, \emph{concave} if the compliment $\R^2_+ \setminus \Omega$ is convex in $\C^n$ and \emph{free} if $\Omega$ is convex and contained in $\R_+^n \subset \R_{\ge 0}^n$ (i.e. disjoint from the coordinate axes). Finally, $\Omega$ is \emph{rational} if it is a moment polytope in the sense of Definition \ref{def:moment_polytope} (i.e. a polytope with rational vertices).
\end{definition}

A fundamental fact in this paper is that a convex rational domain $X$ can be compactified to toric surfaces $Y$ by collapsing the boundary $\partial Y$ so that it becomes the associated ample divisor $A$ of $Y$. More precisely, we have the following result.

\begin{lemma} Let $\Omega$ be a rational, convex domain polytope with toric variety $(Y_\Omega,A_\Omega)$ and toric surface $X_\Omega$. Then there is a $\mathbb{T}^n$-equivariant symplectomorphism
\[Y_\Omega \setminus \on{supp}(A_\Omega) \simeq \intr{X}_\Omega\]
\end{lemma}

\begin{proof} Let $\mu:Y_\Omega \to \R_{\ge 0}^n$ and $\nu:X_\Omega \to \R_{\ge 0}^n$ denote the moment maps of $Y_\Omega$ and $X_\Omega$. Define $\Omega^\circ$ to be the complement $\Omega \setminus (\partial\Omega \cap \R_+^n)$. First note that $\Omega^\circ$ is the moment image of both $Y_\Omega \setminus \text{supp}(A_\Omega)$ under $\mu$ and $\intr{X}_\Omega$ under $\nu$. For $\intr{X}_\Omega$ this is clear, and true for any convex domain. 

For $Y_\Omega$, write the associated ample divisor as $A_\Omega = \sum_F a_F \cdot D_F$. By examination of Definition \ref{def:associated_ample}, we see that $a_F = 0$ if and only if $F$ is on a plane passing through $0$. Since $\Omega = K \cap \R_{\ge 0}^n$ for some convex $K$, we know that $\Omega$ intersects each coordinate hyperplane $H_i = \{x \in \R^n|x_i = 0\}$ along a single face $F_i$ and every other face $F_i$ is not contained in a plane containing the origin (essentially by convexity). Thus $a_{F_i} = 0$ for each $i$ and $a_F \neq 0$ for every other face $F$. Thus $Y_\Omega \setminus \text{supp}(A_\Omega) = \mu^{-1}(\Omega^\circ)$.

\vspace{-15pt}
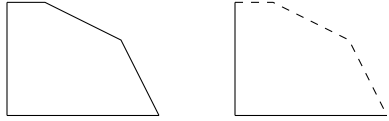
\begin{figure}[h]
\caption{Moment polytopes for $Y_\Omega$ and $Y_\Omega\setminus D_\Omega$}
\begin{center} \label{fig:poly}
\vspace{5pt}
\begin{tikzpicture}[scale=0.5]
\node (o) at (0,0){};
\node (a) at (4,0){};
\node (b) at (0,3){};

\node (v1) at (1,3){};
\node (v2) at (3,2){};
\node (v3) at (4,0){};

\draw (o.center) to (a.center);
\draw (o.center) to (b.center);

\draw (b.center) to (v1.center) to (v2.center) to (v3.center);

%-----------------

\node (o) at (0+6,0){};
\node (a) at (4+6,0){};
\node (b) at (0+6,3){};

\node (v1) at (1+6,3){};
\node (v2) at (3+6,2){};
\node (v3) at (4+6,0){};

\draw (o.center) to (a.center);
\draw (o.center) to (b.center);

\draw[dashed] (b.center) to (v1.center) to (v2.center) to (v3.center);
\end{tikzpicture}
\end{center}
\end{figure}

Now that we have shown that $X^\circ_\Omega$ and $Y_\Omega \setminus \text{supp}(A_\Omega)$ have the same moment images, we just apply an open version of Delzant's theorem, e.g. the result of Kershon-Lerman \cite[Thm 1.3]{kl}. Note that, in that result, there is a homological obstruction $\mathfrak{o}$ to the equivalence of two spaces with the same moment image
\[\mathfrak{o} \in H^2(X^\circ_\Omega;R) = H^2(Y_\Omega \setminus \text{supp}(A_\Omega);R)\]
for some abelian group $R$. This obstruction necessarily vanishes since $X^\circ_\Omega$ is contractible.
\end{proof}

Note that (essentially by definition) a moment polytope $\Omega \subset \R^n$ is equivalent to a convex, rational polytope $\R_{\ge 0}^n$ by scalings and $\GL_n(\Z)$-affine maps if and only if $\Omega$ has a smooth vertex.

\begin{example} Considering ellipsoids $X_\Omega=E(a,b)$ and the corresponding toric varieties $\pr(1,a,b)$, we recover the (well-)known compactifications
$$\pr^2\setminus H=B(1)^\circ\text{ and }\pr(1,a,b)\setminus H=E(a,b)^\circ$$
where $H=\mO(1)$ is a hyperplane section in each variety respectively.
\end{example}

\subsection{Axioms of $\calg$ for toric surfaces} \label{subsec:axioms_of_calg} This section is devoted to proving that the algebraic capacities of toric surfaces satisfy a set of important formal properties.

\begin{thm} \label{thm:axioms_of_calg} Let $Y_\Omega$ be a projective toric surface with moment polytope $\Omega$ and associated ample divisor $A_\Omega$. Then the $k$th algebraic capacity satisfies the following axioms.
\begin{itemize}
	\item[(a)] (Scaling/Affine Maps) If $S > 0$ is a constant and $T:\Z^2 \to \Z^2$ is an affine isomorphism, then
	\[\calg_k(Y_{S\Omega}, A_{S\Omega}) = S \cdot \calg_k(Y_\Omega,A_\Omega) \quad\text{and}\quad \calg_k(Y_{T\Omega}, A_{T\Omega}) = \calg_k(Y_\Omega,A_\Omega)\]
	\item[(b)] (Inclusion) If $\Omega \subset \Delta$ is an inclusion of moment polytopes, then
	\[\calg_k(Y_\Omega,A_\Omega) \le \calg_k(Y_\Delta,A_\Delta)\]
	\item[(c)] (Blow Up) If $\pi:Y_{\widetilde{\Omega}} \to Y_\Omega$ is a birational toric morphism with one exceptional fiber $E$ and associated ample divisor $A_{\widetilde{\Omega}} = \pi^*A_\Omega - \epsilon E$ for $\epsilon > 0$ small, then
	\[\calg_k(Y_{\widetilde{\Omega}},A_{\widetilde{\Omega}})\leq\calg_k(Y_\Omega,A_\Omega)\]
	\item[(d)] (Embeddings) If $X \subset \R^4$ be a star-shaped domain that symplecically embeds into $Y_\Omega$, then
	\[\cech_k(X) \le \calg_k(Y_\Omega,A_\Omega)\]
	\item[(e)] (Domains) If $\Omega$ is a (convex or free) domain polytope and $X_\Omega$ is the associated toric domain, then
	\[\cech_k(X_\Omega) = \calg_k(Y_\Omega,A_\Omega)\]
\end{itemize}
Furthermore, axioms (a)-(e) uniquely characterize the algebraic capacities $\calg_k$ on toric surfaces.
\end{thm}

\begin{proof} We will need some of these properties to prove the others, so we must proceed in a particular order. We first prove (a), (c) and (e) which are mutually independent. We then apply these properties to acquire (b) and apply Theorem \ref{thm:main} to acquire (d).

\vspace{5pt}

\emph{(a) - Scaling/Affine Maps.} First, note that a toric domain transforms as $Y_{S\Omega} = Y_\Omega$ and the divisor transforms as $A_{S\Omega} = S \cdot A_\Omega$. So the scaling axiom follows from Definition \ref{def:alg_cap}. 

Next, we must show invariance if $T$ is either linear or a translation. If $T \in \text{GL}_2(\Z)$ is linear, then $T$ is an automorphism on the Lie algebra $\R^2 \simeq \mathfrak{t}^2$ of $\mathbb{T}^2$ induced by a group automorphism of $\mathbb{T}^2$. Thus $(Y_\Omega,A_\Omega)$ and $(Y_{T\Omega},A_{T\Omega})$ are identical after pulling back by this automorphism, and the algebraic capacities must agree. If $T$ is a translation then $Y_\Omega = Y_{T\Omega}$ and $A_\Omega = A_{T\Omega} + R$ where $R$ is a principle divisor determined by $T$. On the other hand, $A_\Omega \cdot D$ for a divisor $D$ depends only on the divisor class of $A_\Omega$, and so invariance follows from Definition \ref{def:alg_cap}.

\vspace{5pt}

\emph{(c) - Blow Up.} Let $D$ be a nef $\Q$-divisor on $Y$ that achieves the optimum defining $\calg_k(Y,A)$, i.e.
\[\calg_k(Y,A)=D\cdot A\text{ and }\chi(D)\geq k+1\]
Consider the proper transform $\pi^*D$ of $D$ on $\wt{Y}$, which is nef. This has $\chi(\pi^*D)=\chi(D)\geq k+1$. Therefore, the algebraic capacities satisfy
\[\calg_k(Y_{\widetilde{\Omega}},A_{\widetilde{\Omega}})\leq\pi^*D\cdot A_{\widetilde{\Omega}} = \pi^*D\cdot (\pi^*A_\Omega - \epsilon E) = D \cdot A_\Omega = \calg_k(Y_\Omega,A_\Omega)\]

\vspace{5pt}

\emph{(e) - Domains.} This is simply a restatement of Theorem 4.15 and Theorem 4.18 of \cite{bwo}, which state that if $\Omega$ is is a convex domain polytope or a convex free polytope, then
\begin{equation} \label{eqn:axiom_domain_pf} \cech_k(X_\Omega) = \inf_{D\in\on{nef}(Y_\Omega)_\Q}\{D\cdot A_\Omega:h^0(D) \ge k + 1\}\end{equation}
This result is phrased in terms of $\Q$-divisors, and also uses global sections instead of the Euler characteristic. However, since $Y_\Omega$ is toric we have Demazure vanishing.
\begin{lemma}[{\cite[Thm.~9.3.5.]{cls}}] \label{lem:nefq_van} Suppose $Y$ is a toric surface and $D$ is a nef $\Q$-divisor. Then
$$h^p(D) = 0\text{\textnormal{ for all $p>0$}}$$
\end{lemma}
\noindent Thus $h^0(D) = \chi(D)$. Moreover, we have the following Lemma (see \cite[Lem. 2.1]{bwr}).

\begin{lemma} \label{lem:round_down_divisor} Let $D$ be a nef $\Q$-divisor on $Y_\Omega$. Then there exists a nef $\Z$-divisor with
\[h^0(D') = h^0(D) \qquad A_\Omega \cdot D' \le A_\Omega \cdot D\]
\end{lemma}

\begin{proof} Without loss of generality assume $D$ is a torus-invariant divisor and let $D=\sum a_F D_F$. Consider the round-down of $D$, defined by
$$\lf D\rf:=\sum\lf a_F\rf D_F$$
which is a $\Z$-divisor with $P_D \cap \Z^n = P_{\lf D\rf} \cap \Z^n$. The difference $D-\lf D\rf$ is effective and so $\lf D\rf\cdot A\leq D\cdot A$. Unfortunately, $\lf D\rf$ may not be nef. 

\vspace{3pt}

To fix this, we modify $\lf D\rf$ to a nef divisor $D'$ by translating some of the hyperplanes $H_F = \{x|\langle u_F,x\rangle \ge -\lf a_F\rf\}$ (see \S \ref{subsec:toric_surfaces}) for $\lf D\rf$ inwards if necessary. (Here we are using the nef criterion discussed in \S \ref{subsec:toric_surfaces}.) This is equivalent to subtracting some integer multiple of the prime divisor $D_F$ and hence only reduces the area. We must also translate each hyperplane only until it meets a lattice point in $P_{\lf D\rf}$ for $\lf D\rf$, so that $h^0(D') = h^0(\lf D\rf)$. Note that every lattice point in $\Z^n$ is in one of the translates of $H_F$, for each $F$, so we can always perform this translation process while ensuring that $P_{D'} \cap \Z^n = P_{D} \cap \Z^n = P_{\lf D\rf} \cap \Z^n$. In particular, $h^0(D') = h^0(D)$.
\end{proof}

Lemmas \ref{lem:nefq_van} and \ref{lem:round_down_divisor} together imply that the following two infima are equal.
\[\inf_{D\in\on{nef}(Y_\Omega)_\Q}\{D\cdot A_\Omega:h^0(D) \ge k + 1\} = \inf_{D\in\on{nef}(Y_\Omega)_\Z}\{D\cdot A_\Omega:\chi(D) \ge k + \chi(\mathcal{O}_Y)\}\]
In view of (\ref{eqn:axiom_domain_pf}) and Definition \ref{def:alg_cap}, we conclude that $\cech_k(X_\Omega) = \calg_k(Y_\Omega,A_\Omega)$.

\vspace{5pt}

\emph{(b) - Inclusion.} Let $\Omega \subset \Delta$ be an inclusion of moment polytopes. By the application of an affine transformation $T:\Z^2 \to \Z^2$ to both $\Omega$ and $\Delta$, we may assume that $\Omega$ and $\Delta$ are in $(0,\infty)^2 \subset \R^2$, and thus are convex free polytopes. By (e) and the fact that $X_\Omega \subset X_\Delta$, we have
\[\calg_k(Y_\Omega,A_\Omega) = \cech_k(X_\Omega) \le \cech_k(X_\Delta) = \calg_k(Y_\Delta,A_\Delta)\]

\vspace{5pt}

\emph{(d) - Embeddings.} Let $X \to Y_\Omega$ be a symplectic embedding of a star-shaped domain. If $Y_\Omega$ has no singularities (i.e. no singular fixed points), this is simply Theorem \ref{thm:main}. Otherwise, since $X$ is a smooth and compact, its image misses the singular fixed points. Thus we can take a toric resolution $\pi:Y_{\widetilde{\Omega}} \to Y_\Omega$, where $\widetilde{\Omega}$ is acqurired from $\Omega$ by cutting off small triangles from the singular corners. For sufficiently small cuts, $Y_{\widetilde{\Omega}}$ inherits an embedding $X \to Y_{\widetilde{\Omega}}$ and thus we have
\[\cech_k(X) \le \calg_k(Y_{\wt{\Omega}},A_{\wt{\Omega}}) \le \calg_k(Y_\Omega,A_\Omega)\]
Here we apply either the blow up axiom (c) or the inclusion axiom (b). 

\vspace{5pt}

\emph{Uniqueness.} Finally, to argue that these axioms uniquely determine $\calg_k$, let $d^\text{alg}_k$ be another family of numerical invariants satisfying axioms (a)-(e). The blow up and inclusion axioms imply that $\calg_k$ and $d^\text{alg}_k$ agree if and only if they agree on all polytopes $\Omega$ such that $Y_\Omega$ is non-singular. Any such polytope is equivalent to a domain polytope by scaling and affind transformation, so by (a) we merely need to check those polytopes. Then (e) implies that the invariants must agree for those polytopes. \end{proof}

\begin{remark} Theorem \ref{thm:main} and the blow up property (c) can be used together to give an indendent proof of the upper bound of the ECH capacities by the algebraic capacities in Theorem 4.15 of \cite{bwo}. However, we are not aware of a proof that establishes a lower bound which is not essentially equivalent to the one provided in \cite{bwo}. A fundamentally different proof could potentially shed light on an approach to Conjecture \ref{conj:equality}.  
\end{remark}

\subsection{Embeddings to toric surfaces} \label{subsec:emb_to_toric} We now prove the main applications of the paper, which are easy consequences of the axioms in Theorem \ref{thm:axioms_of_calg}. We start by showing that the algebraic capacities are complete obstructions for embeddings of the interiors of concave toric domains into a toric surfaces, in terms of $\cech$ and $\calg$.

\begin{thm} \label{thm:concave_to_tor_surface} Let $X_\Delta$ be a concave toric domain and let $(Y_\Omega,A_\Omega)$ be a projective toric surface with a smooth fixed point. Then
\[X^\circ_\Delta \text{ symplectically embeds into }Y_\Omega \quad \iff \quad \cech_k(X_\Delta) \le \calg_k(Y_\Omega,A_\Omega)\]
\end{thm}

\begin{proof} Suppose that $X^\circ_\Delta \to Y_\Omega$ is a symplectic embedding, and let $X_i$ be an exhaustion of $X^\circ_\Delta$ by star-shaped domains. Then
\[\cech_k(X_\Delta) = \lim_{i \to \infty} \cech_k(X_i) \le \calg_k(Y_\Omega,A_\Omega)\] 

On the other hand, suppose that $\cech_k(X_\Delta) \le \calg_k(Y_\Omega,A_\Omega)$. Since $Y_\Omega$ has a torus fixed point, we can scale by an $S > 0$ and apply an affine map $T:\Z^2 \to \Z^2$ so that $TS(\Omega)$ is a convex domain polygon for convex toric domain $X_{TS(\Omega)}$. Applying axioms (a) and (e) of Theorem \ref{thm:axioms_of_calg}, we acquire
\[\cech_k(X_{S\Delta}) \le \calg_k(Y_{TS(\Omega)},A_{TS(\Omega)}) = \cech_k(X_{TS(\Omega)})\]
Now we apply a well-known result \cite[Thm. 1.2]{cg} of Cristofaro-Gardiner stating that a concave toric domain $X_{S\Delta}$ embeds into a convex toric domain $X_{TS(\Omega)}$ if and only if the ECH capacities of $X_{S\Delta}$ are bounded by those of $X_{TS(\Omega)}$. Thus we acquire a symplectic embedding
\[X_{S\Delta}^\circ \to X_{TS(\Omega)}^\circ \subset Y_{TS(\Omega)} \simeq Y_{S\Omega}\]
Since scaling the moment image merely scales the symplectic form accordingly, we thus acquire a symplectic embedding $X_\Delta^\circ \to Y_\Omega$. \end{proof}

\begin{cor} \label{co:gromov_monotonicity} Let $\Omega \subset \Delta$ be an inclusion of moment polygons, each of which has a smooth vertex. Then the Gromov widths satisfy
\[c_G(Y_\Omega) \le c_G(Y_\Delta)\]
In particular, $c_G$ is monotonic with respect to inclusions of the moment polytope for smooth toric surfaces.
\end{cor} 

\begin{proof} Let $B(r) \to Y_\Omega$ be a symplectic embedding of a closed ball of symplectic radius $r$. Then by the embedding axiom and inclusion axiom in Theorem \ref{thm:axioms_of_calg}, we have 
\[\cech_k(B(r)) \le \calg_k(Y_\Omega,A_\Omega) \le \calg_k(Y_\Delta,A_\Delta)\]
Thus by Theorem \ref{thm:concave_to_tor_surface}, we have an embedding $B^\circ(r) \to Y_\Delta$ of the open ball of symplectic radius $r$, so $r \le c_G(Y_\Delta)$. Taking the sup over all such embeddings $B(r) \to Y_\Omega$ yields $c_G(Y_\Omega) \le c_G(Y_\Delta)$.
\end{proof}

In fact, we can prove a more general result than Corollary \ref{co:gromov_monotonicity}. Namely, given a moment image $\Xi$ for a concave toric domain and a symplectic manifold $Y$, define the \emph{$\Xi$-width} $c_\Xi(Y)$ by
\[c_\Xi(Y) := \text{sup} \{r \; : \; X_{r\Xi} \text{ embeds symplectically into }Y\}\]
Then by the same argument as in Corollary \ref{co:gromov_monotonicity}, we have the following result.

\begin{cor} \label{cor:Xi_width_monotonicity} Let $\Omega \subset \Delta$ be an inclusion of moment polygons, each of which has a smooth vertex. Then
\[c_\Xi(Y_\Omega) \le c_\Xi(Y_\Delta)\]
\end{cor}

\begin{remark} It seems that one can also execute the proof of Corollary \ref{cor:Xi_width_monotonicity} using only the fact that a ball $B(r)$ embeds into $X_\Omega$ if and only if it embeds into $Y_\Omega$ (see \cite[Thm~1.2]{chmp}) and the inclusion axiom (b) of Theorem \ref{thm:axioms_of_calg}. However, this would not cover any singular surfaces, and furthermore the stronger Corollary \ref{cor:Xi_width_monotonicity} requires the results of this paper. \end{remark}

A consequence of Theorem \ref{thm:concave_to_tor_surface} is that the $\Xi$-width of a convex toric domain $X_\Omega$ where $\Omega$ has rational slopes agrees with the $\Xi$-width of the toric surface $Y_\Omega$.

\begin{cor} \label{cor:xi_width_agree} Suppose $\Omega$ is a convex domain with rational slopes. Then
$$c_\Xi(X_\Omega)=c_\Xi(Y_\Omega)$$
\end{cor}

\subsection{Gromov width and lattice width}

We use Corollary \ref{co:gromov_monotonicity} to provide a combinatorial upper bound for the Gromov width of a toric surface as conjectured in \cite{ahn}. We recall the definition of the lattice width.

\begin{definition} The \emph{lattice width} $w(\Omega)$ of a moment polytope is defined by
\[w(\Omega) := \underset{l \in \Z^n \setminus 0}{\text{min}}\Big( \underset{p,q \in \Omega}{\text{max}} \; \langle l,p - q\rangle\Big)\]
\end{definition}

\begin{cor} \label{cor:gromov_combo_bound} Let $\Omega$ be a moment polygon with a smooth vertex. Then $c_G(X_\Omega) \le w(\Omega)$.
\end{cor}

\begin{proof} We implement the heuristic argument in \cite[Rmk 3.13]{ahn} rigorously. Let $l \in \Z^2 \setminus 0$ be the vector such that 
\[w(\Omega) = \sup_{p,q \in \Omega} |\langle l,p - q\rangle|\]
We can choose an element $A \in \GL_2(\Z)$ such that $(A^T)^{-1}(l) = e = (1,0)$ is the $x$-basis vector. This implies that
\[\langle e,A(p - q)\rangle = \langle (A^{-1})^Tl,A(p - q)\rangle = \langle l, p - q\rangle = w(\Omega) = w(A\Omega) \]
Thus the lattice width of $A\Omega$ is achieved in the direction of $e$. We can thus fit $A\Omega$ in a rectangle $R$ of width $a_1 = w(\Omega)$ and very large height $a_2 \gg a_1$. Since $A \Omega \subset R$, we apply Corollary \ref{cor:gromov_combo_bound} to acquire the inequality
\[c_G(Y_{\Omega}) = c_G(Y_{A\Omega}) \le c_G(Y_{R})\]
On the other hand, $Y_R \simeq \P^1(a_1) \times \P^1(a_2)$ and since $a_2 \gg a_1$, we have that $c_G(Y_R) = a_1 = w(\Omega)$.
\end{proof}

\end{document}